\documentclass[12pt]{cpres}

\title{Cusp Volumes of Alternating Knots on Surfaces} 
\author{Brandon Bavier}
\thanks{The author is supported by NSF Grants DMS-1404754, DMS-1708249, DMS-2004155, and the Dr. Williams L. Harkness Endowed Fellowship}

\address{Department of Mathematics, Michigan State University, East Lansing, MI 48824, USA}
\email{bavierbr@msu.edu} 

\graphicspath{{Pictures/}}

\begin{document}

\begin{abstract}
We study the geometry of hyperbolic knots that admit alternating projections on embedded surfaces in closed 3-manifolds. We show that, under mild hypothesis, their cusp area admits two sided bounds in terms of the twist number of the alternating projection and the genus of the projection surface. As a result, we derive diagrammatic estimates of slope lengths and give applications to Dehn surgery. These generalize results of Lackenby and Purcell  about alternating knots in the 3-sphere. 

Using a result of Kalfagianni and Purcell, we  point out  that alternating knots on surfaces of higher genus, 
can have arbitrarily small cusp density, in contrast to alternating knots on spheres whose cusp densities are bounded away from zero due to Lackenby and Purcell. 

\end{abstract}

\maketitle
\setcounter{tocdepth}{1}
\tableofcontents
\section{Introduction}

Given a knot $K$ in a closed 3-manifold $Y,$ we get a 3-manifold by looking at the complement, $Y\setminus K$. By Thurston ~\cite{TDGT}, this manifold admits a canonical decomposition into geometric pieces. Often, $Y\setminus K$ is hyperbolic---that is, it admits a Riemannian metric of constant curvature $-1$ and has finite volume. By Mostow rigidity ~\cite{MPRIGID}, geometric measurements in the hyperbolic metric are also topological invariants of $Y\setminus K$.

In $S^3$, we  study knots through their projections onto a 2-sphere, and in this setting,
 a goal of modern knot theory is to relate the geometry of the knot complement to diagrammatic quantities and invariants. 
Lackenby has shown that the volume of  hyperbolic alternating links in $S^3$ admits two sided bounds in terms of the twist number of any alternating diagram ~\cite{VHLC}, while Agol and Thurston, in the appendix of ~\cite{VHLC}, improved the upper bound for all hyperbolic links. This upper bound was then improved by Adams ~\cite{BPB}. Futer, Kalfagianni, and Purcell have extended the lower bounds to more general classes of links ~\cite{DFVJP,SLCS, SHKT}. Other diagrammatic lower volume bounds have been obtained by looking at the guts of state surfaces of semi-adequate links ~\cite{Guts, HSAL}.

In the study of knots through their projections, it is also  natural and important to search for diagrammatic criteria for recognizing hyperbolicity. 
A classical result of Menasco ~\cite{CISAKC} states that prime alternating  knots  in $S^3$ are either hyperbolic or $(2, p)$ torus knots. Several authors have looked at generalizations of alternating links and were able to extract geometric information from
their projections.
Adams, for example studied knots and links projected onto a torus ~\cite{TAKL}, and found conditions for when these were hyperbolic. Hayashi ~\cite{Hayashi} and Ozawa ~\cite{NTGAK} also extended these for knots on higher genus closed surfaces in $S^3$. Adams, Albors-Riera, Haddock, Li, Nishida, Reinoso, and Wang ~\cite{HypLinksTS} studied  alternating links in thickened surfaces, and Champanerkar, Kofman, and Purcell ~\cite{GeoBAltLinks} looked at alternating links in thickened tori arising from biperiodic alternating links. Howie and Purcell studied a different generalization of alternating knots, called weakly generalized alternating knots ~\cite{WGA} (see Definition ~\ref{def:wgak}), which  are projected onto a general surface in any 3-manifold. They showed that such knots  are hyperbolic and found lower bounds for their volume, under the condition that the regions of the projection are disks.

An important geometric invariant of hyperbolic knots, which is much less understood than the volume of the knot complement, is the cusp volume.  Futer, Kalfagianni, and Purcell ~\cite{CAFM}  gave two-sided diagrammatic bounds for cusp volumes of closed 3-braids and of 2-bridge knots.  More recently,
Lackenby and Purcell gave  two-sided linear  bounds for the cusp volume of usual alternating knots  in terns of   twist numbers of alternating projections~\cite{CVAK}.

In this paper,  we generalize the work of Lackenby and Purcell for weakly generalized alternating knots  in irreducible 3-manifolds. Our main result is the following theorem  that gives two sided bounds on the cusp volume of such knots, under mild restrictions. The terms
involved in the statement of the theorem are defined in Section 2.
These restrictions are given in terms  of two quantities, the edge representativity $e(\pi(K),F)$ and the representativity $r(\pi(K),F)$, that are naturally associated to weakly generalized alternating knots. We note that the hypothesis on $r(\pi(K),F)$ is automatically satisfied in the case that the projection surface is incompressible in $Y$ (Definition ~\ref{def:rep}).

\begin{restatable}{thm}{mainclaim}
\label{thm:main}Let $Y$ be a closed, irreducible 3-manifold and  $F\subset Y$ a closed surface with $\chi(F)\le0$
and such that  $Y\setminus N(F)$ is atoroidal.
Suppose that $K$ is  a weakly generalized alternating knot with a projection $\pi(K)\subset F$ that  is weakly twist reduced,  and we have $e(\pi(K),F)>2$, and $r(\pi(K),F)>4$.
Then $Y\setminus K$ is hyperbolic, and we have

$$ \frac{5.8265514\times 10^{-7}}{2\  (120-\chi(F))^6} \ (\tw_{\pi}(K)-\chi(F))\le CV(K)\le \frac{(360 \tw_{\pi}(K) - 3\chi(F))^2}{\tw_{\pi}(K)},$$

where $CV(K)$  is the cusp volume of $K$ and  $\tw_{\pi}(K)$  the twist number of ${\pi}(K)$.
\end{restatable}

In the case that $F$ is a Heegaard torus in a Lens space (including in $S^3$) Theorem ~\ref{thm:main} has the following  topological application.

\begin{restatable}{thm}{gromov}
\label{thm:gromov} Suppose that  $Y$ is a Lens space or $S^3$ and that $F$ is a Heegaard torus in $Y$.
For any closed 3-manifold $M$ with  Gromov norm at least  $2.3984\times 10^{23}$,
there are finitely many weakly generalized alternating knots  $K\subset Y$, with projections $\pi(K)$ as in Theorem ~\ref{thm:main},  and slopes $\sigma$ such that $M$ is obtained by  Dehn filling $Y\setminus K$ along $\sigma$.
\end{restatable}
Theorem ~\ref{thm:gromov}, which generalizes ~\cite[Theorem 1.3]{ETS}, is related to Problem 3.6 (D) in Kirby's list ~\cite{95problems} which asks if there are any 3-manifolds which can be obtained by surgery along an infinite number of distinct knots in $S^3$. While Osoinach ~\cite{MOSIK} proved that
such manifolds exist, by Theorem ~\ref{thm:gromov}, when we restrict to certain weakly generalized alternating knots  we can't find such manifolds of large Gromov norm.
We will prove this theorem  in Section 1.3.

\subsection{Slope Lengths on the Cusp}
Given a hyperbolic knot complement $M=Y\setminus K$ there is a well defined notion of  its  maximal cusp $C$.
The boundary of C, denoted by $\bndry C$, inherits a Euclidean structure from the hyperbolic metric.
Given an essential simple closed curve $\sigma \subset \partial C$ (a.k.a. a slope), the length of $\sigma$ is  the Euclidean length of the unique geodesic in the homotopy call of $\sigma$. We well denote this length by  $\ell(\sigma)$.
\begin{defn}
\label{def:meridian}
If $K$ is a  knot in a 3-manifold $Y$, take a regular torus neighborhood of $K$. The meridian $m$ is the curve that bounds a disk in $Y$ that intersects $K$ exactly once. 
\end{defn}

The 6-Theorem  of Agol and Lackenby ~\cite{BEDF,WHDS}, implies that every hyperbolic knot in a non-hyperbolic 3-manifold has the meridian with length at most 6. 
 Adams, Colestock, Fowler, Gillam, and Katerman   ~\cite{CSBSS} for hyperbolic knots in $S^3$, the meridian has length  strictly less than 6. The techniques of  ~\cite{CSBSS}  where extended in
~\cite{GESS} to obtain upper bounds on slope lengths  of hyperbolic knots in 3-manifolds in terms of spanning surfaces of the knots.

By Thurston ~\cite{TDGT}, the length of any slope with respect to the maximal cusp is bounded below by 1, but more effective lower bounds are hard to find. For knots in $S^3$ with projections that have 
 at least 145
 crossings per region, Purcell showed that  the length of non-meridian slopes is bounded below by a constant times the number of twist number ~\cite{CSCD}. Futer, Kalfagianni, and Purcell  ~\cite{CAFM} found a similar bound
 for two-bridge knots. This was  generalized  by Lackenby and Purcell in ~\cite{CVAK} for all alternating hyperbolic knots in $S^3$.
Here we will discuss slope length bounds for alternating knots on surfaces.

The cusp area, denoted by $\mathrm{Area}(\bndry C)$, is the Euclidean area of $\partial C,$ 
 and it is also twice the cusp volume $CV(K)$.
Given a non meridian slope  $\sigma \subset \partial C$,  as in ~\cite{DSNCM}, 
we see that, 
\[\ell(\mu)\ell(\sigma)\ge \mathrm{Area}(\bndry C)\Delta(\mu,\sigma),\]
where $\Delta(\mu,\sigma)$ is the geometric intersection number of our slopes on $\bndry C$.

Using Theorem ~\ref{thm:main} and a result of Burton and Kalfagianni ~\cite{GESS}, we obtain the following corollary, the proof of which is given in Section 3.

\begin{restatable}{cor}{slopelength}
\label{cor:slopelength}
Let $Y$ be a closed, irreducible 3-manifold and  $F\subset Y$ a closed surface with $\chi(F)\le0$
and such that  $Y\setminus N(F)$ is atoroidal.
Suppose that $K$ is  a weakly generalized alternating knot with a projection $\pi(K)\subset F$ that  is weakly twist reduced,  and we have $e(\pi(K),F)>2$, and $r(\pi(K),F)>4$.
Let $\mu$ be the meridian and $\sigma$ any non-meridian slope on maximal cusp of $Y\setminus K$.
Then, we have
\begin{align*}
\ell(\mu)&\le \frac{720 \ (\tw_\pi(K)-\chi(F))}{\tw_\pi(K)}\\
\ell(\sigma)&\ge 
\left( \frac{5.8265514\times 10^{-7}}{720 \  (120-\chi(F))^6}\right) \  \tw_\pi(K),
\end{align*}
where $\tw_\pi(K)$ is the twist number of $\pi(K)$.
\end{restatable}

\subsection{Dehn Filling Applications}
The bounds on slopes from Corollary ~\ref{cor:slopelength} have applications for Dehn surgery along knots that admit weakly generalized alternating projections. By the 6-Theorem of Agol ~\cite{BEDF} and Lackenby ~\cite{WHDS}, if $M$ is a hyperbolic 3-manifold, then filling $M$ along a slope $\sigma$ of length strictly greater than 6, on a torus cusp, produces a hyperbolic manifold. If the length is bigger than $2\pi$, then we get a bound on the volume, by Futer, Kalfagianni, and Purcell:

\begin{thm}[Theorem 1.1 ~\cite{DFVJP}]
\label{thm:FKP}
Let $M$ be a complete, finite-volume hyperbolic manifold with a single cusp $C.$ Let $\sigma$ be a slope on $\bndry C$ with length $\ell(\sigma)$ greater than $2\pi,$ and let  $M(\sigma)$  denote  the 3-manifold obtained by Dehn filling $M$ along $\sigma$. Then, $M(\sigma)$ is hyperbolic and we have
\[\mathrm{vol}(M(\sigma)) \ge \left(1-\left(\frac{2\pi}{\ell(\sigma)}\right)^2\right)^{3/2}\mathrm{vol}(M).\]
\end{thm}

Combining this with Theorem ~\ref{thm:main}, Theorem ~\ref{thm:wga_main} (from ~\cite{WGA}), and Theorem ~\ref{thm:FKP} and Thurston ~\cite{TDGT}, we get the following result:

\begin{thm}
Let $K$ be a weakly generalized alternating knot on a surface $F\neq S^2$ in a closed manifold $Y$ with conditions as in Theorem ~\ref{thm:main}. Suppose that
\[\tw_{\pi}(K)> \left(2.32928\times 10^{10}\right)  \ \left(120-\chi(F)\right)^6.\]
Then, any manifold $M$ obtained from non-meridional surgery along $K$ is hyperbolic and 
\[\frac{v_8}{4}\left(\tw_{\pi}(K)-\chi(F)\right)\le\frac{1}{2}\mathrm{vol}(Y\setminus K)\le \mathrm{vol}(M) < \mathrm{vol}(Y\setminus K).\]
Here $v_8\approx 3.66386$ is the volume of a regular hyperbolic ideal octahedron.
\end{thm}

To get the leftmost inequality, we can first use Lemma ~\ref{lem:disk_regions} to get that the regions of $F\setminus K$ are all disks, which then allows us to use Theorem ~\ref{thm:wga_main} to get a lower bound on volume.  Then, for the center inequality, by taking our number of twist regions to be high enough, we get that the lengths of any slopes must be at least $2\pi$, and so we can apply the Theorem ~\ref{thm:FKP}. Finally, the rightmost inequality comes from Thurston and the fact that volume must decrease under surgery ~\cite{TDGT}.

\subsection{Weakly Generalized Alternating Knots in Lens spaces}

While the geometry of these generalized knots has several features in common with alternating knots, there are still several key differences. In particular, because we can change which projection surface we are working with, we must often consider the topology of $F$ in our formulas. For example, in the case of volume, every alternating knot in $S^3$ has an upper and lower bound on volume based solely on the number of twist regions of a diagram on $S^2$, by work of Lackenby, and improved by Agol and Thurston ~\cite{VHLC}.  Kalfagianni and Purcell obtained a similar result for weakly generalized alternating knots on Heegaard tori of Lens spaces. Specifically,  ~\cite[Corollary 1.5]{ALSVB} states the following: 
  Suppose $K$ has a weakly generalized alternating projection $\pi(K)$ to a Heegaard torus $F$ in $Y=S^3$, or in $Y$ a lens space. Suppose that $\pi(K)$ is twist reduced, the regions of $F-\pi(K)$ are disks, and the representativity satisfies $r(\pi(K),F)>4$. Then $Y-K$ is hyperbolic, and we have
$\mathrm{vol}(Y\setminus K) \le 10v_3\tw_{\pi}(K),$
  where $v_3$ is the volume of an ideal hyperbolic tetrahedron. We will  now use this result to prove Theorem ~\ref{thm:gromov}.

\begin{proof}{\rm {(of  Theorem ~\ref{thm:gromov})}}
We proceed as in ~\cite{ETS}, and will need a theorem of Cooper and Lackenby ~\cite{DSNCM}, which states that, for any $\epsilon > 0$ and any closed 3-manifold $M$, there are at most finitely many cusped hyperbolic manifolds $X$ and slopes $\sigma$ on  cusps $\bndry X$, with length at least $2\pi + \epsilon$, such that $M$ is obtained by Dehn filling $X$ along the slopes $\sigma$. 
Set $B(F)=\displaystyle{ \left( \frac{5.8265514\times 10^{-7}}{720 \  (120-\chi(F))^6}\right)},$ the quantity appearing in the lower bound for $\ell(\sigma)$ in  Corollary ~\ref{cor:slopelength}.

In addition, we will use Corollary ~\ref{cor:slopelength} in the case that $F$ is a torus, so $\chi(F)=0$ by a direct calculation we get
$\frac{B(F)}{720} \ge 2.710139\times 10^{-22}$. This means that any non-meridian slope $\sigma$ on the maximal cusp of $X=Y\setminus K$ we have  $\ell(\sigma)\geq 2.710139\times 10^{-22} \tw_{\pi}(K)$.

Now let $K$ be a weakly generalized alternating knot $K$ on the torus $F$, with $r(\pi(K),F)>4$ and $e(\pi(K),F)>2$ as in the statement of the corollary.
Also let  $M$ be a 3-manifold with Gromov norm at least $2.3984\times 10^{23}$, and suppose that $M$ is obtained by Dehn filling from  $X = Y\setminus K$, along a slope $\sigma$.
 As the Gromov norm does not increase under Dehn filling, $X$ must also have Gromov norm at least $2.3984\times 10^{23}$. As hyperbolic volume and Gromov norm are proportional, 
 this tells us that
$$v_3 2.3984\times 10^{23} \le \mathrm{vol}(X) \le 10 v_3\tw_{\pi}(K),$$
with the last inequality coming from ~\cite[Corollary, 1.5]{ALSVB} mentioned above.
So then $\tw_{\pi}(K)$ must be at least $2.3984\times 10^{22}$.
 
Finally, by Corollary ~\ref{cor:slopelength} in the case $F$ is a torus, we have the length $\ell(\sigma)$  is at least $(2.3984\times 10^{22}) (2.710139\times 10^{-22}) > 6.5$. If we take $\epsilon = 6.5-2 \pi$, then using the theorem of Cooper and Lackenby, we are done, and get the result of Theorem ~\ref{thm:gromov}.
\end{proof}

The proof of Theorem ~\ref{thm:gromov} doesn't work for weakly  generalized alternating knots on  higher genus surfaces. This is due to a lack of an upper bound on volume in terms of  $\tw_{\pi}(K)$.
Indeed, Kalfagianni and Purcell ~\cite{ALSVB} found a family of weakly generalized alternating knots in $S^3$ with  constant $\tw_{\pi}(K)$ on a Heegaard surface of genus two, whose volumes are arbitrarily large. Since the cusp density of a knot
$K\subset Y$ is the ratio of the cusp volume of $K$ over the volume of $Y\setminus K$,
as a consequence of this, and the upper bound of Theorem ~\ref{thm:main}, one obtains the following.

\begin{cor}[Corollary 1.5 ~\cite{ALSVB}] \label{density} There exist
weakly generalized alternating knots in $S^3$ with arbitrarily small cusp density.
\end{cor}

Corollary ~\ref{density} shows that weakly generalized alternating knots even in $S^3$ are geometrically genuinely different than usual alternating knots.
Indeed,  due to Lackenby and Purcell ~\cite{CVAK}, ordinary alternating knots are known to have cusp densities that are bounded below uniformly away from zero.

\subsection{Organization and outline}
The remainder of this paper is split into six sections. In Section 2, we introduce several relevant definitions and constructions, including weakly generalized alternating knots, twisted checkerboard surfaces, and cusp area. 
The twisted surfaces are obtained by  appropriate modifications of the checkerboard spanning surfaces associated with a weakly generalized alternating knot projection.
We finish the section with the statement of Theorem 
~\ref{thm:essential}, that, under appropriate hypotheses, the twisted checkerboard surfaces are $\pi_1$-injective in the knot complement. This theorem and results we obtain during the course of its proof,  are key ingredients for the proof of Theorem ~\ref{thm:main}. The  proof
Theorem 
~\ref{thm:essential} occupies
Sections 4- 7 of the paper.

In Section 3 we prove Theorem ~\ref{thm:main}, assuming Theorem ~\ref{thm:main} and Theorem ~\ref{thm:homarcs}. The proof of the later result  is also postponed till Section 7, as it relies on the machinery developed for the proof of 
Theorem ~\ref{thm:main}. 
There are several geometric results we use from Lackenby and Purcell ~\cite{CVAK}, especially those relating essential surfaces to geodesics and area. While we will mention the relevant results and their proofs, including why they hold for our general case, for the complete picture it is recommended to look at the original paper.

In Section 4, we talk more in depth about weakly generalized alternating knots, focusing particularly on what happens to them under augmentation (removing or adding crossings in pairs). In sections 5-7, we generalize ~\cite{ETS}, focusing on how disks in the knot complement might intersect the twisted checkerboard  surfaces. The combinatorics of these intersections are encoded by certain planar graphs. The proof of Theorem ~\ref{thm:essential} relies on a careful analysis 
of properties of these graphs that lead us to
understand how a potential disk might intersect the twisted surface, as well as other surfaces arising from the knot projection. 

\subsection{Acknowledgments}
The author wishes to thank his advisor, Efstratia Kalfagianni, for guidance and advice. This material, which will be part of the author's upcoming PhD dissertation at Michigan State University, is based on research partially supported by NSF Grants DMS-1404754, DMS-1708249, DMS-2004155, and the Dr. Williams L. Harkness Endowed Fellowship.

\section{Definitions}

Before we begin, we'll introduce some definitions and tools we will use throughout this paper. We will start with a generalization of alternating knots, called weakly generalized alternating knots. Then, we will introduce the surfaces we will be examining. Finally, we will summarize a few of the proofs from the original paper that can be used for generalized knots with little to no changes.

\subsection{Weakly Generalized Alternating Knots}

These definitions are covered in more detail, and slightly more generality, in ~\cite{WGA}. Let $Y$ be a compact, orientable, irreducible 3-manifold, with a closed, orientable surface $F$, which we call a \textit{projection surface}. Given a knot $K\subset F\times I\subset Y$, we can get a projection $\pi(K)$ by flattening $F\times I$ down to just $F$, and keeping track of where the knot goes. We call $\pi(K)$ a \textit{generalized diagram}. As we are not working with $F=S^2\subset S^3$, we need to modify some of the usual definitions we have for knots. In the setting of $F=S^2$, one uses the fact that all curves in $S^2$ bound disks to define primeness. Because we do not have this property for general surfaces, we modify as below:

\begin{defn}
\label{def:wprime}
We say that $\pi(K)\subset F$ is \textit{weakly prime} if, whenever $D\subset F$ is a disk with $\bndry D$ intersecting $\pi(K)$ transversely exactly twice, then either
\begin{itemize}
	\item $F=S^2$, and either $\pi(K)\cap D$ is a single arc or $\pi(K)\cap (F\setminus D)$ is;
	\item $F$ has positive genus, and $\pi(K)\cap D$ is a single arc.
\end{itemize}
\end{defn}

In order to generalize the proofs that follow, we will want to make sure that our knots are sufficiently complex. To do this, we will introduce two conditions on the knot:

\begin{defn}
\label{def:rep}
The \textit{edge representativity} $e(\pi(K),F)$ is the minimum number of intersections between $\pi(K)$ and any essential curve on $F$. The \textit{representativity} $r(\pi(K),F)$ is the minimum number of intersections between $\pi(K)$ and a compression disk of $F$, taken over all compressions disks of $F$. If there are no essential curves, we set $e(\pi(K),F)$ to be $\infty$, while if there are no compression disks, we set $e(\pi(K),F)$ and $r(\pi(K),F)$ to be $\infty$. 
\end{defn}

While the definition of a weakly generalized alternating knot only uses the representativity,, many of our proofs will also need the edge representativity to be high enough to deal with some edge cases. Now, though, we have enough to give a definition for a weakly generalized alternating knot.

\begin{defn}
\label{def:wgak}
Let $F\subset Y$ be a projection surface as above. Then the diagram $\pi(K)$ on $F$ of a knot $K$ is \textit{reduced alternating} if
\begin{enumerate}[(1)]
	\item $\pi(K)$ is alternating on $F$,
	\item $\pi(K)$ is weakly prime,
	\item $\pi(K)\cap F \neq \emptyset$,
	\item $\pi(K)$ has at least one crossing on $F$.
\end{enumerate}
If, in addition, we also have
\begin{enumerate}[(1)]
	\setcounter{enumi}{4}
	\item $\pi(K)$ is checkerboard colorable on $F$,
	\item $r(\pi(K),F)\ge 4$
\end{enumerate}
then $\pi(K)$ is a \textit{weakly generalized alternating diagram}, and $K$ is a \textit{weakly generalized alternating knot}.
\end{defn}

These knots have been studied in several other papers, as well as knots that satisfy subsets of this definition. These were introduced by Howie and Rubenstein ~\cite{HR_WGA} in $S^3$ as generalizations of Hayashi ~\cite{LADPosGen} and Ozawa ~\cite{NTGAK}, and generalized further to other manifolds $Y$ by Howie and Purcell ~\cite{WGA}.

If we ignore the final two conditions of the above definition, and just look at reduced alternating, we get as a subset generalized alternating knots in $S^3$, as studied by Ozawa in ~\cite{NTGAK}, with the additional restrictions that $Y=S^3$ and $\pi(K)$ is strongly prime, which implies weakly prime. These knots are also checkerboard colorable, although. the representativity is only at least 2 ~\cite[Theorem 2.2]{NTGAK}. If we don't require $\pi(K)$ to be weakly prime, and let $F$ to be a Heegaard torus, we get toroidally alternating knots, as studied by Adams ~\cite{TAKL}. More broadly, this category also includes alternating knots on a Heegaard surface $F$, as studied by Hayashi ~\cite{LADPosGen}. This also fits the alternating projection of a knot $\pi(K)$ onto it's Turaev surface $F$ which also has the property that $F\setminus \pi(K)$ are disks ~\cite{JPGraphSurf}. By contrast, for reduced alternating and weakly generalized alternating knots, $F$ does not need to be a Heegaard surface, nor does $F\setminus \pi(K)$ need to be disks. In addition, the requirement that $r(\pi(K),F)$ be large enough also often guarantees that our diagram will be sufficiently complicated.

In this paper, we will narrow our definition by requiring the representativity $r(\pi(K),F)$ to be strictly greater than 4, and adding in that the edge representativity $e(\pi(K),F)$ is at least 4. These restrictions will allow us to later rule out certain troublesome cases caused by working on a surface with genus. We will also define twist reduced in this general case, following Howie and Purcell's definition ~\cite[Definition 6.3]{WGA}:

\begin{defn}
\label{def:wtwred}
A reduced alternating knot diagram $\pi(K)$ on $F$ is \textit{weakly twist reduced} if every disk $D$ in $F$, with $\bndry D$ meeting $\pi(K)$ in exactly two crossings, either contains only bigon faces of $F\setminus \pi(K)$ or $F\setminus D$ contains a disk $D'$, where $\bndry D'$ meets $\pi(K)$ in the same two crossings and bounds only bigon disks. 
\end{defn}

We will need the following result of Howie and Purcell  that is that allows to determine when a knot (or link) is hyperbolic based on from a generalized alternating projection
and bounds the volume with diagrammatic quantities. 

\begin{thm}[Theorem 1.1, ~\cite{WGA}]
\label{thm:wga_main}
Let $\pi(L)$ be a weakly generalized alternating projection of a link $L$ onto a generalized projection surface $F$ in a 3-manifold $Y$. Suppose $Y$ is compact, orientable, irreducible, and has empty boundary. Finally, suppose $Y\setminus N(F)$ is atoroidal. If $F$ has genus at least one, the regions in the complement of $\pi(L)$ on $F$ are disks, and the representativity $r(\pi(L),F)>4$, then
\begin{enumerate}
	\item $Y\setminus L$ is hyperbolic
	\item $Y\setminus L$ admits two checkerboard surfaces that are esssential and quasifuchsian.
	\item The hyperbolic volume of $Y\setminus L$ is bounded below by a function of the twist number of $\pi(L)$ and the Euler characteristic of $F$:
	\[\mathrm{vol}(Y\setminus L) \ge \frac{v_8}{2} (\tw_{\pi}(L)-\chi(F))\]
\end{enumerate}
\end{thm}

In order to use this, we must show that, if $e(\pi(K), F) \ge 4$, then $F\setminus \pi(K)$ are disks:

\begin{lem}
\label{lem:disk_regions}
Let $\pi(K)$ be a weakly generalized alternating projection of $K$ onto a generalized projection surface $F$. If $e(\pi(K), F)\ge 4$, then the regions of $F\setminus \pi(K)$ are all disks.
\end{lem}

\begin{proof}
Suppose not. Then one of the regions of $F\setminus \pi(K)$ either contains an annulus whose core is essential or meets itself at a crossing. If the region has an annulus, take $\gamma$ to be the core. Then $\gamma$ doesn't intersect $\pi(K)$, and so $e(\pi(K),F)=0$, a contradiction. If a region meets itself at a crossing, let $\gamma$ be the curve that meets this crossing and then connects back to itself in the region. Then $\gamma$ intersects $\pi(K)$ exactly twice (at just the crossing), so $e(\pi(K),F)\le 2$, a contradiction. Thus if $e(\pi(K),F)\ge 4$, then the regions of $F\setminus \pi(K)$ are all disks.
\end{proof}

\subsection{Twist Number Remarks}
The twist number of a knot diagram is the number of twist regions in the diagram $D$, often labeled $\tw(D)$. In the traditional case, working with projection surface $S^2$ inside $S^3$, we can look at the minimum twist number over all diagrams, and get the twist number of the knot, $\tw(K)$. This is an invariant of the knot. If we require our diagram to be prime, reduced, twist reduced, and alternating, we get that $\tw(D)=\tw(K)$, and so the twist number of this diagram is an invariant. 

In the case of knots projected onto a surface that isn't $S^2$, however, we don't necessarily have an invariant. Howie was able to find two different weakly generalized alternating diagrams of the $9_29$ knot on a Heegaard torus in $S^3$ with two different twist numbers.

It is not known, however, if the twist number of a twist-reduced diagram of a weakly generalized alternating knot is an invariant in other settings. For example, when we look at knots in thickened surfaces $S\times I$, we know the crossing number of a reduced alternating diagram is an invariant ~\cite{CrossNoAltK}. As such, it's natural to ask if the twist number is invariant here. 

\subsection{Twisted Surfaces}

Now that we have a knot, we can look at different spanning surfaces for $K$. Two immediate surfaces we get are the checkerboard surfaces, which we can color red and blue, and label $R$ and $B$ respectively, both of which are essential in $Y\setminus\pi(K)$. While we could, as in the first half of ~\cite{CVAK}, proceed with just $R$ and $B$, this will give us a lower bound based on both the twist number as well as the number of crossings. Instead, we will follow the second half, and use the twisted checkerboard surfaces of Lackenby and Purcell ~\cite{ETS}, which we will now describe how to construct.

Let $K$ be a weakly generalized alternating hyperbolic knot on a surface $F$. Before we build the twisted surfaces for our knot, we must first augment our link. Technically, we will be working with the projected diagram $\pi(K)\subset F$ as opposed to the knot itself, however we will abuse notation and refer to this diagram as $K$. 

(1) Let $N_{tw}$  be some number. While we will eventually set this number to be at least 121, for definitions it can be as large or as small as we want. (2) Around each twist region with at least $N_{tw}$ crossings in it, add a crossing circle, and call the augmented link $L$. (3) For each encircled twist region, remove pairs of crossings until only one or two remain. We call this new link $L_2$. Note that $Y\setminus L$ is homeomorphic to $Y\setminus L_2$ by twisting along the neighborhoods of the crossing circles. (4) Let $K_2$ be the diagram where we have removed all of the crossing circles from $L_2$. 

\begin{figure}
	\centering
		\def\svgwidth{\columnwidth}
		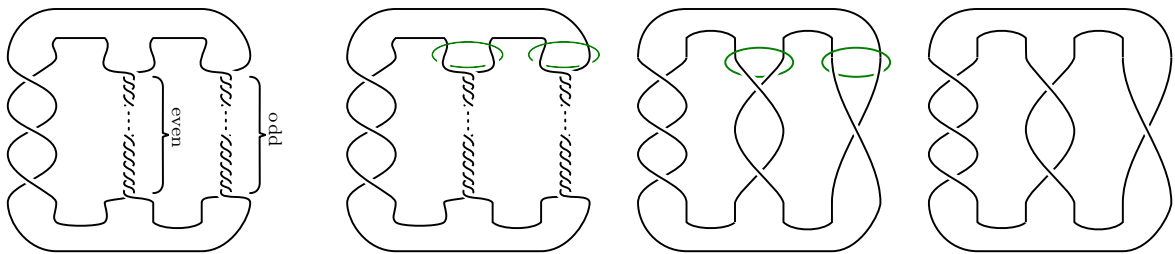

		\caption{The steps of augmenting a link. From left to right, a knot $K$ with three twist regions, only two of which are highly twisted; adding crossing circles to get $L$; removing all but one or two crossings to get $L_2$; Removing crossing circles to get $K_2$.}
\end{figure}

This augmented knot, $K_2$, remains weakly generalized alternating, which we shall prove in section four. As such, as long as $Y\setminus K$ is hyperbolic, so is $Y\setminus K_2$.

In order to get our twisted surfaces for $K$, we will start with surfaces for $K_2$, and work our way back up the chain. As $K_2$ is weakly generalized alternating, it has two checkerboard surfaces on $F$, which we can color blue and red. Next, when we go back to $L_2$, that is, when we put our crossing circles back in, each crossing circle will intersect either the red or the blue surface. Furthermore, when we look at $Y\setminus L_2$, each crossing circle will have a regular neighborhood that intersects either the red surface or the blue surface in a meridian disk. So then define the blue and red surfaces in the exterior of $L_2$ to be the blue and red surfaces of $K_2$ punctured by the crossing circles of $L_2$. We will denote these by $B_2$ and $R_2$, respectively.

Now we need to send $B_2$ and $R_2$ back to $Y\setminus L$. Adding back in the crossings is the same as twisting along the crossing circles. Most of $B_2$ and $R_2$ will twist to give us the usual checkerboard surface of $L$. Things will change, however, where the surfaces meet the crossing circle. If we removed $2n$ crossings from a twist region to get to $L_2$, we will need to twist a full $n$ times around in order to get them back. That is, the meridian curves where our surface intersects our surfaces will go to $\pm 1/n$ curves on the boundary of the crossing circle, with sign chosen appropriately.

Finally, to get back to $Y\setminus K$, we need to fill in the crossing circles with a meridian Dehn filling. So in order to get our twisted surfaces, we will need to complete them in a way that makes sense with this filling. Note that each (meridional) cross-section of the crossing circle intersects the punctured surfaces $2n$ times on the boundary. We then connect opposite points by an interval running between them. When we do this for the whole crossing circle, we will either be attaching a single annulus or two M\"{o}bius bands, depending on the value of $n$. In either case, however, we now have two immersed surfaces, which we can still color blue and red, and which we call \textit{twisted checkerboard surfaces}. We will denote them $S_{B,2}$ and $S_{R,2}$ appropriately.

\begin{figure}
	\centering
		\def\svgwidth{\columnwidth}
		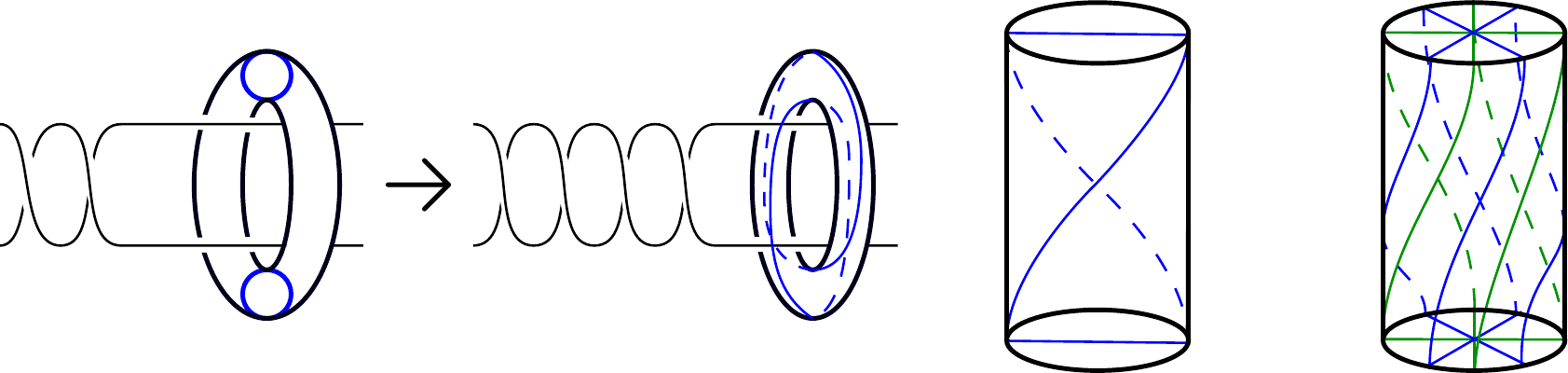

		\caption{\textit{(Left)} $L_2$ with the torus around the twist region. Blue circles mark where the torus punctures the blue surface. \textit{(Center Left)} Twisting around the torus gives us back $L$. In this case, we are only adding in one full twist for simplicity. \textit{(Center Right)} To complete the twisted surface, attach opposite sides of the torus with a band, as represented on top and bottom. Here this amounts to adding an M\"{o}bius band. \textit{(Right)} When we have more twists, things can become more complicated. In addition, we might be adding two annuli instead.}
\end{figure}

Our goal with these surfaces is to prove the following:

\begin{restatable}{thm}{essentialthm}
\label{thm:essential}
Let $f:S_{B,2}\to Y\setminus K$ be the immersion of our blue twisted surface $S_{B,2}$ into our knot complement $Y\setminus K$, where $Y$ has no boundary components. Then this immersion is $\pi_1$-injective and boundary
$\pi_1$-injective provided that $N_{tw}\ge 121$.
\end{restatable}
Switching the roles of blue and red surfaces Theorem ~\ref{thm:essential} implies that the immersed surfaces  $S_{R ,2}$  are  $\pi_1$-injective and boundary
$\pi_1$-injective, provided that $N_{tw}\ge 121$.

The proof of Theorem  ~\ref{thm:essential}  will be given in Section 6. However in Section 3, we will assume Theorem  ~\ref{thm:essential}, to derive Theorem ~\ref{thm:main}.

\subsection{Cusp Volume and Area}

Let $M$ be a finite-volume hyperbolic manifold with $\bndry M$ a collection of tori. Then the ends of $M$ are of the form $T^2\times[1,\infty)$. As $M$ is hyperbolic, we have some covering map $\rho:\bH^3\to M$. Then if we look at the pre image of these ends, they look like horoballs in $\bH^3$.

For each end, we can choose a horoball representative $H_i$. The image of this representative, $\rho(H_i)=C_i$ is called a \textit{cusp} of $M$. In actuality, we get a family of such cusps as we shrink or expand $H_i$. In our case (a knot complement), there is only a single cusp --- the place where we removed the knot. When we only have a single cusp, we can expand it as much as possible, until it becomes tangent to itself. In $\bH^3$, this is expanding each horoball pre-image until they begin to become tangent to each other. Once we have done this, we get what is called the \textit{maximal cusp} of $M$. In the case of a single cusp, this expansion is unique. We can, of course, generalize this to multi-cusped manifolds, however it is not quite as neat. Depending on the order we expand the cusps, we can get several maximal cusps, and each can have different properties.

Now, once we have our choice of cusp, $C$, we can start analyzing it. As our manifold $M$ has some (hyperbolic) geometric structure to it, $C$ will naturally inherit a geometry from it. And because we know the pre-image of $C$ is a horoball in $\bH^3$, it must inherit a hyperbolic structure and it's boundary must be Euclidean. We can then ask about the geometric properties of the cusp.

\begin{defn}
\label{def:cvol}
The \textit{volume} of a cusp $C$, $\mathrm{vol}(C)$, is the Euclidean volume of $C$. The \textit{area} of a cusp, $\mathrm{area}(\bndry C)$, the Euclidean area of $\bndry C$.
In the case that $M=Y\setminus K$ we will use $CV(K)$ to denote the volume of the maximal cusp of $M$. 
\end{defn}

As we can view the boundary of a cusp as $T^2\times{1}$, by some simple calculus, we can see that $\frac{1}{2}\mathrm{area}(\bndry C) = \mathrm{vol}(C)$. As the volume can be found from the area, we will often focus on just finding a single value, usually the area, and seeing what that gives us.

We have several methods to help us calculate the cusp area. First, if we have a triangulation of $M$, we can find the area by direct computation, as has been implemented in SnapPy ~\cite{SnapPy}. When we work with fully augmented links, where all crossings of a knot have been removed and crossing circles have been added in, we can completely determine the geometry by a circle packing. In this case, we can use the circle packing to find the cusp area ~\cite{LNES}.

\begin{defn}
\label{def:cshape}
Let $\mu$ be the meridian of our knot, and $\lambda$ the shortest longitude. Then the lengths of these curves, $\ell(\mu)$ and $\ell(\lambda)$, determine the similarity class of the Euclidean structure on $\bndry C$. We call this similarity class the \textit{cusp shape}.

Knowing the cusp shape can tell us about the area of the cusp, as we have $\mathrm{Area}(\bndry C) \le \ell(\mu)\ell(\lambda)$. 

\end{defn}

Adams, Colestock, Fowler, Gillam, and Katerman found upper bounds of slope lengths and cusp shapes of a knot based on it's crossing number ~\cite{CSBSS}. Burton and Kalfagianni
found upper bounds coming from essential spanning surfaces of knots ~\cite{GESS}. We need to recall their result ~\cite[Theorem 4.1]{GESS} as we will use it in the next section.

 \begin{thm}[Theorem 4.1 ~\cite{GESS}]
\label{thm:gess}
Let $K$ be a hyperbolic knot with maximal cusp $C$, in an irreducible 3-manifold $Y$. Suppose that $S_1$ and $S_2$ are essential spanning surfaces in $M=Y\setminus K$, and let $i(\bndry S_1, \bndry S_2)\neq 0$ denote the minimal intersection number of $\bndry S_1$ and $\bndry S_2$ in $\bndry C$. Let $\ell(\mu)$ and $\ell(\lambda)$ denote the lengths of the meridian and the shortest longitude of $K$, respectively. Then:
\[\ell(\mu)\le \frac{6|\chi(S_1)|+6|\chi(S_2)|}{i(\bndry S_1, \bndry S_2)} \quad \quad \ell(\lambda)\le 3|\chi(S_1)|+3|\chi(S_2)|,\]
and
\[\mathrm{area}(\bndry C) \le \frac{18(|\chi(S_1)|+|\chi(S_2)|)^2}{i(\bndry S_1, \bndry S_2)}.\]

\end{thm}

Note that the authors in ~\cite{CSBSS} state their result for links in $S^3$. However, as noted in
While the original theorem of Burton and Kalfagianni is stated in $S^3$, the results work for any manifold $Y$, as long as our knot is hyperbolic.

When looking at weakly generalized alternating knots, we can look at the two checkerboard surfaces, obtained by ``checkerboard coloring'' $F\setminus \pi(K)$, and attaching same-colored adjacent pieces by a twisted band at crossings.  In the case of the weakly generalized alternating knots we are studying, these surfaces are essential ~\cite[Theorem 1.1]{WGA}. Furthermore, with the proof that our twisted surfaces are essential in Theorem ~\ref{thm:essential}, we get another upper
bound based on $\tw_{\pi}(K)$ and $\chi(F)$, which we state and prove in Theorem ~\ref{thm:main} at the end of the next section.

\section{Cusp Volumes of Weakly Generalized Alternating Knots}

Now, assuming Theorem ~\ref{thm:essential}, which shows twisted surfaces are essential, we will prove Theorem ~\ref{thm:main}, and get a lower bound for the cusp area based on the twist number of our knot and $\chi(F)$. For the lower bound we will the approach of Lackenby and Purcell ~\cite{CVAK} to our setting. We note that some of the results we  reference from  ~\cite{CVAK}, and used in there for knot complements in $S^3$,  apply immediately to knot complements in any closed $Y$.
For the convenience of the reader we will restate them and we will review them and their proofs briefly below. 

 As a reminder of our goal, we'll restate Theorem ~\ref{thm:main} here, but will postpone it's proof until the end of this section:

\mainclaim*

Before we start the proofs, there is an important definition that we will use throughout. Most geodesics in a hyperbolic manifold are bi-infinite, but we will want to be able to talk about their lengths. As such, we use the following definition to get around this complication:

\begin{defn}
\label{def:geolength}
Let $\gamma$ be a (bi-)infinite geodesic in a hyperbolic surface $R$, with both ends in a horoball neighborhood of the cusps of $R$, $H$. Then $\gamma\cap H$ is a collection of closed intervals (if $\gamma$ enters and exits $H$) along with two helf-open intervals (the ends). The \textit{length of $\gamma$ with respect to $H$} is the hyperbolic length of $\gamma$ minus the two half-open intervals.

If $\alpha$ is an arc with boundary on $\bndry M$, we first find a geodesic $\gamma$ such that $\alpha$ is homotopic to $\gamma$ in $M$, and define the length of $\alpha$ with respect to $H$ to be the length of $\gamma$ with respect to $H$.
\end{defn}

The next result concerns an estimate of geodesic arcs on hyperbolic surfaces. We will apply the result specifically to our twisted surfaces, but it works in more general settings.

\begin{thm}[Theorem 2.7 ~\cite{CVAK}]
\label{cvak:2.7}
Let $S$ be a (possibly disconnected) finite-area hyperbolic surface. Let $H_S$ be an embedded horoball neighborhood of the cusps of $S$. Let $k = \mathrm{Area}(H_S)/\mathrm{Area}(S)$ and let $d>0$. Then there is a collection of at least
\[\frac{(ke^d-1)\pi}{(e^d-1)(\sinh(d)+2\pi)} \left|\chi(S)\right|\]
embedded disjoint bi-infinite geodesic arcs, each with both ends in $H_S$, and each having length at most $2d$ with respect to $H_S$.
\end{thm}

The next lemma that holds without change deals with estimating the cusp area of a manifold, based on essential arcs of bounded length:

\begin{lem}
\label{cvak:2.8}
Suppose that a one-cusped hyperbolic 3-manifold $M$ contains at least $p$ homotopically distinct essential arcs, each with length at most $L$ measured with respect to the maximal cusp $H$ of $M$. Then the cusp area $\mathrm{Area}(\bndry H)$ is at least $p\sqrt{3}e^{-2L}$.
\end{lem}

\begin{sproof} The proof is similar to that of Lemma 2.8 in ~\cite{CVAK}:
Each arc in our collection has a geodesic representative in $M$, which then lifts to two geodesics in the universal cover $\bH^3$. Look at the ``shadows'' of these geodesics on the lift of the maximal cusp $\bndry H$ (that is, a neighborhood of where the geodesic intersects $\bndry H$). For the maximal cusp to be embedded, each of these neighborhoods must be disjoint, and, because the length is at most $L$, the diameter of this neighborhood must be at least $e^{-L}$. This means that the total area of the disks must be $p \pi e^{-2L}/2$. Finally, combining with a disk packing argument, we get the area of the cusp must be at least $p\sqrt{3}e^{-2L}$.
\end{sproof}

While these two lemmas suggest that we should use them one after the other, we must be careful. The cusp in Theorem ~\ref{cvak:2.7} is a two-dimensional cusp, while the cusp in Lemma ~\ref{cvak:2.8} is three-dimensional. In addition, the geodesic arcs of Theorem ~\ref{cvak:2.7} live on a surface, while the essential arcs of Lemma ~\ref{cvak:2.8} are in the whole manifold. To get from one to the other, we will need the following theorem, the proof of which is postponed till Section 7. The proof occupies subsection 7.2.
\begin{restatable}{thm}{homotopicarcs}
\label{thm:homarcs}
Let $S$ be the disjoint union of our twisted surfaces $S_{B,2} \sqcup S_{R,2}$. Suppose that two distinct essential arcs in the surface $S_{B,2}$ have homotopic images in $Y\setminus K$, but not $S_{B,2}$. Then the two arcs are homotopic in $S_{B,2}$ into the same subsurface associated with some twist region of $K_2$.
\end{restatable}

Here, a \textit{subsurface associated to a twist region} is the intersection of one of the checkerboard surfaces with a regular neighborhood of the twist region.  See Figure
~\ref{subsurface}.

\begin{figure}
	\centering
		\def\svgwidth{.4\columnwidth}
		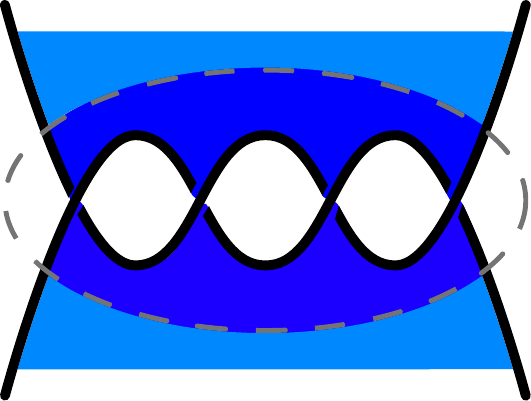

		\caption{The subsurface associated to the twist region. When it comes from checkerboard surfaces, the subsurface can be colored red and blue, with one component (blue in the picture) being two disks attached with M\"{o}bius bands at each crossing, and the other being a twisted disk.}
\label{subsurface}
\end{figure}

 As a corollary to Theorem ~\ref{thm:homarcs} we get the following:

\begin{cor}
\label{cvak:3.4}
Suppose $N\ge 121$. Let $\mathcal{C}$ be a collection of disjoint embedded essential non-parallel arcs in $S_{B,2}\sqcup S_{R,2}$ which are all homotopic in $Y\setminus K$. Then the number of arcs in $\mathcal{C}$ is at most $2(3N-2)$.
\end{cor}

\begin{proof}
By Theorem ~\ref{thm:homarcs}, all such arcs in $\mathcal{C}$ that lie in $S_{B,2}$ can be homotoped in $S_{B,2}$ to lie in a subsurface associated with a twist region of $K_2$. This homotopy keeps theses arcs disjoint, essential, and non-parallel. Because of this, we can use the fact that the number of disjoint, essential, and non-parallel arcs in a subsurface is at most $3c-2$, where $c$ is the number of crossings in the twist region associated to that subsurface. Then, as we are working in $K_2$, $c\le N$, so there are at most $3N-2$ such arcs in the blue twisted surface. By switching out blue for red, we get that there are also at most $3N-2$ such arcs in $S_{R,2}$. Adding these together, we get that there are at most $2(3N-2)$ such arcs total.
\end{proof}

From Corollary ~\ref{cvak:3.4} we can then limit how many disjoint arcs we get from Theorem ~\ref{cvak:2.7}. 

In addition to these, ~\cite[Proposition 3.5]{CVAK} gives us an embedded cusp in our twisted surfaces with a lower bound on area, which we will use with Theorem ~\ref{cvak:2.7}:

\begin{lem}[Proposition 3.5 ~\cite{CVAK}]
\label{cvak:3.5}
If $S$ is the disjoint union of our twisted surfaces with an immersion $f:S\to Y\setminus K$, and $H$ is the maximal cusp for $Y\setminus K$, then there is an embedded cusp $H_S$ for $S$ continued in $f^{-1}(H)$ with area at least $2^{5/4}\tw_\pi(K)$, as long as $K$ is neither the figure-eight knot or the $5_2$ knot.
\end{lem}

We can find such a cusp by pleating $S$ (such a pleating exists because our twisted surfaces are essential by Theorem ~\ref{thm:essential}) and giving it a hyperbolic structure, and then letting the embedded cusp be the union of the cusps of the two twisted surfaces. The lower bound on area then comes from carefully examining the boundaries of our twisted surfaces. By resolving any intersections between the boundaries, we will get a number of curves with length the same as the union of the two boundaries. The exact number is based on the intersection number of our surfaces; in our case, $S_{B,2}$ and $S_{R,2}$ intersect $2\mathrm{cr}(K_2)\ge \tw_\pi(K)$ times. Then, as long as $K$ isn't the figure-eight knot or the $5_2$ knots, by Adams ~\cite{WaistSize2}, we get that each curve must have length at least $2^{5/4}$. Putting this together, we then get the desired lower bound on area.

Finally, we will need the following lemma computing the Euler characteristic of the disjoint union of our twisted surfaces:

\begin{lem}\label{lem:eulchr}
Suppose $K$ is a weakly generalized alternating Knot on a projection surface $F$. Let $\mathrm{tw_N}(K)$ denote the number of twist regions of $\pi(K)$ with at least $N$ crossings, and let $\mathrm{cr}(K_2)$ be the number of crossings of the diagram of $\pi(K_2)$, the knot where we remove all but 1 or 2 crossings from twist regions with more than $N$ crossings in even pairs. Then the Euler characteristics of the blue and red twisted surfaces satisfy
$$|\chi(S_{R,2}) + \chi(S_{B,2})| = \mathrm{cr}(K_2) + 2 \mathrm{tw_N}(K) - \chi(F).$$
\end{lem}

\begin{proof}
Note that the union of the two (non-twisted) checkerboard surfaces in $K_2$ will have each crossing arc of $K_2$ appear twice, once in the red surface, and once in the blue surface. Then, by splitting the crossing arcs of $K_2$ into two homotopic arcs and attaching the blue regions to one arc and the red regions to the other, we get that the Euler characteristic for the disjoint union of these non-twisted surfaces is $\mathrm{cr}(K_2)-\chi(F)$. To obtain $R_2$ and $B_2$, the punctured surfaces, we remove two disks from each crossing circle, and get $|\chi(B_2) + \chi(R_2)| = \mathrm{cr}(K_2)+2\tw_N(K)-\chi(F)$. Then, finally, to obtain the twisted surfaces, we connect these punctures by annuli or M\"{o}bius bands. This doesn't change the Euler characteristic, and so we get our desired result.
\end{proof}

We are now ready to give the proof of Theorem ~\ref{thm:main}.

\begin{proof}
We will begin with the upper bound, using Theorem ~\ref{thm:gess} (from ~\cite{GESS}) to almost immediately get our result. We know that $|\chi(S)|= \mathrm{cr}(K_2) + 2\tw_N(K)-\chi(F) \le 120\tw_{\pi}(K)-\chi(F)$. Then, note that, by construction, our twisted surfaces must intersect at least once for every twist region in our projection, so $i(\bndry S_{B,2}, \bndry S_{R,2}) \ge \tw_{\pi}(K)$. Finally, as $\chi(S_{R,2})$ and $\chi(S_{B,2})$ are both negative, we get 
\[|\chi(S_{B,2}| + |\chi(S_{R,2}|= |\chi(S_{B,2}) + \chi(S_{R,2})| \le 120\tw_{\pi}(K) -\chi(F).\]
So, combining with Theorem ~\ref{thm:gess}, we get:
\[\mathrm{Area}(\bndry C) \le \frac{18 (|\chi(S_{B,2})| + |\chi(S_{R,2})|)^2}{i(\bndry S_{B,2}, \bndry S_{R,2})} \le \frac{18(120\tw_{\pi}(K)-\chi(F))^2}{\tw_{\pi}(K)},\]

which finishes the proof of the upper bound.

To proceed with the proof of the lower bound define

$$A(F) = \frac{18^{1/4}}{5776 \pi^4}\;\;  \frac{1}{2^{3/4}\pi (120-\chi(F))^6 + 4(120-\chi(F))^5}.$$
To prove the lower bound of $CV(K)$ in Theorem ~\ref{thm:main} we will show 
 that
$$A(F)\ (\tw_{\pi}(K)-\chi(F))\le \mathrm{area}(\bndry C) , \  \ {\rm and} \ \  A(F)\geq  \frac{5.8265514\times 10^{-7}}{(120-\chi(F))^6}.$$
Since  $2CV(K)=\mathrm{area}(\bndry C)$, we obtain the desired  bound.

Let $S$ be the disjoint union of the twisted blue and red surfaces, $S_{B,2}\sqcup S_{R,2}$, with an immersion $f$ into $Y\setminus K$. Then let $H$ is the maximal cusp of $Y\setminus K$, and $H_S$ the embedded cusp for $S$ contained in $f^{-1}(H)$ with area at least $2^{5/4}\tw_\pi(K)$ we get from Lemma ~\ref{cvak:3.5}. Note that $\Cr(K_2)$ is less than or equal to 120 times the number of twist regions with less than 121 twists. Thus $\Cr(K_2)+2\tw_N(K) \le 120\tw_\pi(K)$. Then, by Lemma ~\ref{lem:eulchr}, we have
$$\mathrm{Area}(S) = 2\pi (\Cr(K_2))+2\tw_N(K) - \chi(F)) \le 2\pi (120 \tw_{\pi}(K)-\chi(F)).$$
Recall that 
$$k=\mathrm{Area}(H_S)/\mathrm{Area}(S) \ge \frac{\sqrt[4]{2}}{\pi}\frac{\tw_{\pi}(K)}{120\tw_{\pi}(K)-\chi(F)}.$$
The inequality comes from combining our upper bound on $\mathrm{Area}(S)$ and the lower bound on $\mathrm{Area}(H_S)$ of $2^{5/4}\tw_\pi(K)$ from Lemma ~\ref{cvak:3.5}. Then, let $d=\log(2/k)$, so $ke^d = 2$, $\sinh(d) = \frac{1}{k} - \frac{k}{4}$, and $e^d = \frac{2}{k}$. Note that $\frac{1}{k}\le 2^{-1/4}\pi(120-\chi(F))$, so we can say the following:

\begin{align*}
\frac{2}{k} - 1 &< 2^{3/4}\pi\left(120-\frac{\chi(F)}{\tw_{\pi}(K)}\right) = 2^{3/4}\pi O(F,K)\\
\sinh(d)+2\pi &\le 2^{-1/4}\left(120-\frac{\chi(F)}{\tw_{\pi}(K)}\right)-\frac{k}{4}+2\pi\\
	&< 2^{-1/4} O(F,K) + 2\pi\\
\left(\frac{2}{k}-1\right)(\sinh(d)+2\pi) &< 2^{-1/4} \pi (2^{3/4}\pi O(F,K)^2 + 4 O(F,K)),
\end{align*}

where $O(F,K) = 120 - \frac{\chi(F)}{\tw_{\pi}(K)}$. As a side note, while we will leave in $O(F,K)$ for our calculations, once we choose a projection surface, we immediately get an upper and lower bound: $O(F,K)$ is between $120$ and $120-\chi(F)$.\\

Putting this together, by Theorem ~\ref{cvak:2.7}, we have at least
\begin{align*}
\frac{(ke^d-1)\pi}{(e^d-1)(\sinh(d)+2\pi)} |\chi(S)| &> \frac{2^{1/4}}{2^{3/4}\pi O(F,K)^2 + 4 O(F,K)} |\chi(S)|
\end{align*}
embedded disjoint bi-infinite geodesics, with both ends in $H_S$, of length at most $2d$. Now we can map these arcs into $Y\setminus K$. As $S$ is essential, each arc remains essential under this mapping. Furthermore, the length of these arcs can only decrease, so the upper bound of length $2d$ remains. Some of these arcs may be homotopic now in $Y\setminus K$, however, by Corollary ~\ref{cvak:3.4}, there are at most $2(3N-2)=722$ such arcs. So, dividing out, we have at least
\[\frac{2^{1/4}}{722 (2^{3/4}\pi  O(F,K)^2 + 4 O(F,K))} |\chi(S)|\]
disjoint embedded geodesics with both ends in $H$ and length at most $2d$ with respect to $H$.

Now we can use Lemma ~\ref{cvak:2.8} to say that, if the above number is $p$,
\[ \mathrm{Area}(\bndry H) \ge \sqrt{3}e^{-2 (2d)}.\]
Before we put all the math together, to make the numbers look nicer, we will find an upper bound for $e^{-4d}$:
\begin{align*}
e^{4d} &= \left(\frac{2}{k}\right)^{4}\\
&\le \left(2^{3/4}\pi O(F,K)\right)^{4} = 8\pi^4 O(F,K)^4\\
e^{-4d} &\ge \frac{1}{8\pi^4  O(F,K)^4}
\end{align*}
Now, combining this, with all of our other bounds, we get that our cusp volume satisfies
\begin{align*}
\mathrm{Area}(\bndry H) &\ge \frac{2^{1/4}\sqrt{3}|\chi(S)|}{722 (2^{3/4}\pi O(F,K)^2 + 4 O(F,K))} e^{-4d}\\
	&\ge \frac{18^{1/4}}{722 (2^{3/4}\pi  O(F,K)^2 + 4 O(F,K))} \;\; \frac{1}{8\pi^4 O(F,K)^4}  \left(\tw_{\pi}(K)-\chi(F)\right).
\end{align*}
Because ~\cite{CVAK} gives us a lower bound for the case when $F=S^2$, we will focus on the other cases. When $\chi(F)\neq 2$, $O(F,K)\le 120 - \chi(F)$, so we can replace $O(F,K)$ in our equation with $120-\chi(F)$, and we get a lower bound of $A(F)(\tw_\pi(K)-\chi(F))$, and so are done with the lower bound, and our proof.
\end{proof}

\begin{rem}
In the course of proving the lower bound in Theorem ~\ref{thm:main}, we proved that we can replace $A(F)$ 
a function $E(F,K) $  depending on  the Euler characteristic of $F$ and and $\tw_\pi(K)$. This is
that also 
$$E(F,K) = \frac{18^{1/4}}{5776 \pi^4}\;\;  \frac{1}{2^{3/4}\pi O(F,K)^6 + 4 O(F,K)^5}\ \ {\rm{ where }} \ \ O(F,K)=120-\frac{\chi(F)}{\tw_{\pi}(K)}$$
We can get a better lower bound if we replace $A(F)$ with $E(F,K)$. Below are some value for $E(F,K)$ and $A(F)$ for different values of $\chi(F)$ and $\tw_\pi(K)$, showing that, while as $\chi(F)$ decreases, both our bounds get worse, as $\tw_\pi(K)$, the bound $E(F,K)$ starts to improve. \newline
\begin{center}
\begin{tabular}{c||c|c|c}
$\chi(F)$ & $A(F)$                & $E(F,K)$ when $\tw_\pi(K)=10$ & $E(F,K)$ when $\tw_\pi(K) = 100$\\\hline
$0$    & $2.3059 \times 10^{-19}$ & $2.3059 \times 10^{-19}$      & $2.3059 \times 10^{-19}$        \\
$-2$   & $2.0884 \times 10^{-19}$ & $2.2830 \times 10^{-19}$      & $2.3036 \times 10^{-19}$        \\
$-4$   & $1.8945 \times 10^{-19}$ & $2.2604 \times 10^{-19}$      & $2.3013 \times 10^{-19}$        \\
$-100$ & $6.0903 \times 10^{-21}$ & $1.4272 \times 10^{-19}$      & $2.1941 \times 10^{-19}$
\end{tabular}
\end{center}
\end{rem}

Next we apply Theorem ~\ref{thm:main} to obtain estimates of slope lengths on the maximal cusp of $K$. 
We discuss an upper bound on the length on meridian curves for K and lower length bounds
for non-meridian curves slopes.
\slopelength*

\begin{proof}
As shown in the proof of Theorem ~\ref{thm:main} above, if $S_{B,2}$ and $S_{R,2}$ are our two twisted essential surfaces, then:
\[|\chi(S_{B,2})| + |\chi(S_{R,2})|\le 120\tw_\pi(K) - \chi(F) \le 120(\tw_\pi(K) - \chi(F)),\]
where the right inequality comes from the fact that $\chi(F)\le 0$, so $-\chi(F)\le -120\chi(F)$. Also, $i(\bndry S_{B,2}, \bndry S_{R,2})\ge \tw_\pi(K)$. Using these two results with Burton and Kalfagianni ~\cite{GESS} on upper bounds for slope lengths, we get our desired upper bound on $\ell(\mu)$:
\begin{align*}
\ell(\mu) &\le \frac{6|\chi(S_{B,2})| + |\chi(S_{R,2})|}{i(\bndry S_{B,2}, \bndry S_{R,2})}\\
		  &\le \frac{6(120\tw_\pi(K)-\chi(F))}{\tw_\pi(K)}\\
		  &\le \frac{720(\tw_\pi(K)-\chi(F))}{\tw_\pi(K)}.
\end{align*}
Next, Theorem ~\ref{thm:main} tells us that $A(F)(\tw_\pi(K)-\chi(F))\le \mathrm{Area}(\bndry C)$. Then, as in ~\cite{DSNCM}, 
\[\ell(\mu)\ell(\sigma)\ge \mathrm{Area}(\bndry C)\Delta(\mu, \sigma).\]
Then, using our bound for $\mathrm{Area}(\bndry C)$, and that $\mu$ and $\sigma$ must intersect at least once, we get
\[\ell(\mu)\ell(\sigma) \ge A(F)(\tw_\pi(K)-\chi(F)).\]
Next, we use our upper bound for $\ell(\mu)$ to get:
\[\frac{720(\tw_\pi(K) - \chi(F))}{\tw_\pi(K)} \ell(\sigma) \ge \ell(\mu) \ell(\sigma) \ge A(F)(\tw_\pi(K)-\chi(F)).\]
Finally, we rearrange:
\begin{align*}
\ell(\sigma) &\ge A(F) \tw_\pi(K)-\chi(F)\frac{\tw_\pi(K)}{720(\tw_\pi(K)-\chi(F))}\\
			 &= A(F) \frac{\tw_\pi(K)}{720},
\end{align*}
giving the desired inequality.
\end{proof}

\section{Weakly Generalized Alternating Knots and Augmentation}

 In this section, our main goal is to show that, if we start with a weakly generalized alternating link, with some additional properties, and then augment it by removing crossings, we will still have a weakly generalized alternating link with those same properties.

Let $K$ be a weakly generalized alternating link with a closed, orientable projection surface $F$ in a 3-manifold $Y$. Fix a (twist-reduced) diagram $\pi(K)$ on $F$ with edge representativity, $e(\pi(K),F)\ge 4$,
and representativity $r(\pi(K),F)>$, (as defined in Definition ~\ref{def:rep}). We can then identify twist regions of $\pi(K)$. Let $L$ be the link obtained from $K$ where we add a crossing circle to some number of  twist regions. Also let  $L'$ the link where we remove, in pairs, all but one or two crossings from twist regions with a crossing circle in $L$, and $K'$ the knot obtained from $L'$ by removing the crossing circles.
Our first goal is to prove the following which assures that the altered
knot $K'$ continues to be a weakly generalized alternating and has the same properties as $K$.

\begin{restatable}{prop}{augmentprop}
\label{prop:augment}
If $K$ is a weakly generalized alternating knot with diagram $\pi(K)$ on a projection surface $F$ in a 3-manifold $Y$, with $e(\pi(K),F)\ge 4$ and $r(\pi(K), F)>4$, then $K'$ is a weakly generalized alternating knot on the surface $F$ in $Y$ with $e(\pi(K'),F)\ge 4$ and $r(\pi(K'),F)>4$.
\end{restatable}

 Proposition ~\ref{prop:augment}  allows us to often replace $K$ with the simpler knot $K'.$ The proof of the proposition requires a few lemmas that we will prove next.

\begin{lem}\label{lem:redalt}
$\pi(K')$ is reduced alternating on $F$.
\end{lem}

\begin{proof}
There are four conditions we need to show about $\pi(K')$ for it to be reduced alternating. First, as $\pi(K)$ is alternating on $F$, we must also have $\pi(K')$ alternating on $F$---we removed crossings in pairs, and so could not have introduced any non-alternating into our diagram. Next, we show that $\pi(K')$ is weakly prime, as stated in Definition ~\ref{def:wprime}. If it isn't, then we can find a disk $D\subset F$ such that $\bndry D$ intersects $\pi(K')$ exactly twice and $\pi(K')\cap D$ is not a single embedded arc. We can isotope $D$ to be disjoint from crossing circles. When we put crossings back in to $\pi(K')$, we won't add any crossings to $D$---we only add crossings at crossing circles. But then we have a disk $D$ whose boundary intersects $\pi(K)$ exactly twice with $\pi(K)\cap D$ not a single embedded arc. As $\pi(K)$ is reduced alternating, this is a contradiction.

Finally, we must show $\pi(K')\cap F \neq \emptyset$, and that there is at least one crossing on $F$. However, this follows from construction---$\pi(K')$ still lives on $F$, so the first condition is satisfied. For the second condition, as we never remove all crossings from a twist region, if $\pi(K')$ doesn't have a crossing on $F$, neither will $\pi(K)$, once again a contradiction. So then $\pi(K')$ must be reduced alternating on $F$.
\end{proof}

\begin{lem}
$\pi(K')$ is checkerboard colorable.
\end{lem}

\begin{proof}
As $\pi(K)$ is weakly generalized alternating, we know that $\pi(K)$ must be checkerboard colorable. We can use this coloring to obtain a coloring of $\pi(K')$. Because we remove all but one or two crossings from a region, we can obtain $\pi(K')$ from $\pi(K)$ by removing bigons from the same twist region. These bigons must have the same color, so removing them will not change the checkerboard colorability of the diagram. Using this coloring will then give us a coloring of $\pi(K')$.
\end{proof}

\begin{lem}\label{lem:edgerep}
If $e(\pi(K),F)\ge 4$, then $e(\pi(K'),F)\ge 4$.
\end{lem}

\begin{proof}
Suppose not. Then we can find an essential curve $\ell$ in $F$ that crosses our knot diagram, $\pi(K')$, a minimal number of times---zero, one, two, or three times. By isotopy, we can assume $\ell$ intersects $\pi(K')$ transversely away from crossings. First, we show that the number of intersections can't be odd. Because $\pi(K')$ is checkerboard colorable, we can look at the colored regions $\ell$ crosses through. Every time $\ell$ crosses $\pi(K')$, it must switch colored regions: either from red to blue or from blue to red. If $\ell$ intersects $\pi(K')$ an odd number of times, we would then have a curve that started in one color and ended in the other color. As it is a closed curve and our diagram is checkerboard colorable, this is impossible, so $e(\pi(K'), F)$ must be zero or two.

If $e(\pi(K'), F)=0$, then $\ell$ is disjoint from $\pi(K')$. This means, in particular, that we can isotope $\ell$ to be disjoint from any twist region without changing the number of intersections. Each twist region in $\pi(K')$ has at least one crossing; if $\ell$ enters the twist region through one side, it must exit through the same side in order to not intersect $\pi(K')$. But then, as $\pi(K)$ differs from $\pi(K')$ only in twist regions, when we put the twists back in our diagram, we will have an essential curve that does not intersect $\pi(K)$, contradicting the fact that $e(\pi(K),F)=4$.

If $e(\pi(K'), F) = 2$, then $\ell$ must intersect $\pi(K')$ exactly twice. First, suppose $\ell$ doesn't intersect a twist regions. Then, as before, when we go back to $\pi(K)$, we will have an essential curve intersecting it twice, a contradiction. So then $\ell$ must intersect $\pi(K')$ in a twist region. By the definition of $e(\pi(K'),F)$, we assume that $\ell$ intersects $\pi(K')$ transversely away from crossings. As such, in order for $\ell$ to intersect a twist region, it must intersect the twist region exactly twice, so we will focus on just this region. There are two options for how $\ell$ intersects the strands in this region: either $\ell$ intersects the same strand twice, or $\ell$ intersects both strands once. First, suppose it is the former. Then the arc of $\ell$ intersection the bigon of the twist region bounds a disk with the strand $\ell$ intersects. Using this disk, we can isotope $\ell$ away from the twist region, giving us an essential curve that intersects $\pi(K')$ zero times, a contradiction.

On the other hand, if $\ell$ intersects both strands of the twist region, we may isotope $\ell$ away from any crossing disks. Once we do this, when we put crossings back in to get $\pi(K)$, we will have an essential arc $\ell$ intersecting $\pi(K)$ exactly twice (as $\ell$ does not cross any crossing disks). This is a contradiction, and so we are done.
\end{proof}

We are now ready to give the proof of  Proposition ~\ref{prop:augment} .

\begin{proof}
All but the last part follows directly from the previous lemmas. Thus, all that we need to show is that $r(\pi(K'),F)>4$. As any curve that bounds a compression disk must be essential, we can say, by Lemma ~\ref{lem:edgerep}, that we at least have $r(\pi(K'),F)\ge 4$. So suppose we can find some $\gamma\subset F$ such that $\gamma$ bounds a compression disk and intersects $\pi(K')$ exactly four times. When we put the crossings back to get $\pi(K)$, we must have $\gamma$ intersect more times, and so $\gamma$ must be between a crossing region. However, if $\gamma$ doesn't intersect the knot strands associated to the crossing region, then we have a region with zero crossings, contradicting our construction of $K'$. So at least two of the intersections of $\gamma$ with $\pi(K')$ are part of the same crossing region. But then, no matter how they intersect, we can then homotope $\gamma$ to be outside of the crossing region without introducing any more intersections (but possibly reducing them), as follows. Because we are working on a twist region, there must be a disk on $F$ such that the boundary passes through each of the crossings of our twist region. We can then use part of that disk to homotope away from the compression disk, as shown in the picture above.
\begin{figure}
	\centering
		\def\svgwidth{.75\columnwidth}
		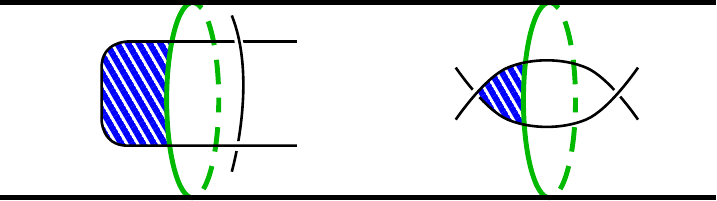

		\caption{Two possible configurations to have $r(\pi(K'),F)\le 4$. On the left, the blue disk can be used to homotope the knot away from the compression disk, reducing the number of intersections. On the right, homotopying along the blue disk will cause the compression disk to not intersect a crossing region.}
\end{figure}
But then, by putting crossings back in, we must have $\gamma$ intersect $\pi(K)$ four or less times, a contradiction to $r(\pi(K),F)>4$, and so we are done.
\end{proof}

Now we specify the knots $K'$ to which we will apply  Proposition ~\ref{prop:augment}.
The construction $K'$ is the same as this in Section 3 of  ~\cite{ETS}  

\begin{itemize}

\item First recall the links $L$, $L_2$ and $K_2$ and the surfaces $B_2$ and $R_2$ constructed from the weakly generalized alternating projection $\pi(K)$ in Subsection. 2.1.

\item Denote by  $L_B$ the link  obtained by  adding in the crossing circle to twist regions  of $\pi(K)$ with $N_{tw}\geq 121$, but only when the crossing circle intersects the blue region.

\item Obtain $L_{B,2}$ by removing pairs of crossings encircled by each crossing circle of $L_B$, leaving either one or 2 crossings.

\item Obtain the knot $K_{B,2}$ by removing crossing. circles from $L_{B,2}$.

\item The surface $B_2$ embedded in the complement of $L_2$ can also be embedded the complement of $L_{B,2}$. The red surface $R_2$ in $S^3 \setminus L{B,2}$ is homeomorphic to the red checkerboard surface of $K$, and to the red checkerboard surface of $K_{B,2}$.
\end{itemize}

We will generalize the  approach of ~\cite{ETS} to prove that $B_2$ and $R_2$ are essential. Ai we will discuss in the remaining sections many of the arguments in  ~\cite{ETS} require little to no modifications to work in the general case. This is because several of the these arguments take place in only a small portion of either our ambient manifold $Y$ or our projection surface $F$, or come mostly from definition.

We will finish up this section by using Proposition ~\ref{prop:augment} to prove two simple, but important, facts about $K_{B,2}$, as we now know some nice properties stay under augmentation. 
The first result is the following.

\begin{lem}
\label{lem:opagree}
Label the regions of the complement of the diagram $\pi(K_{B,2})$ blue or red, depending on whether they meet the blue or red surface.
\begin{enumerate}[(1)]
	\item The blue regions on opposite sides of a crossing of $L_{B,2}$ cannot agree.
	\item The red regions on opposite sides of a crossing of $L_{B,2}$ cannot agree.
	\item The two blue regions that meet a single crossing circle of $L_{B,2}$ cannot agree.
\end{enumerate}
\end{lem}

\begin{proof}
For the most part, the argument follows this of the proof of ~\cite[Lemma 3.2]{ETS}.  
To start, we assume that blue regions on the opposite side of a crossing do agree. Then we can create a simple closed curve on our surface $F$ by drawing an arc in the blue region from one side of the crossing to the other, then drawing an arc across the red region near this crossing. Now we have a simple closed curve that intersects $\pi(K_{B,2})$ exactly twice. If this curve bounds a disk in $F$, then we contradict the fact that $\pi(K_{B,2})$ is weakly prime. If it doesn't bound a disk, then it is an essential curve intersecting our diagram exactly twice, contradicting $e(\pi(K_{B,2}),F)\ge 4$. A similar proof works to prove the second statement by switching the blue and red regions.

Now, to show the third statement, assume that two blue regions meeting a single crossing circle do agree. Then we can draw an arc entirely in the blue region from one side of the crossing circle to the other. We then finish it into a simple closed curve by attaching the arc of intersection of the crossing disk. As above, we have a curve intersecting our diagram exactly twice, contradicting either $\pi(K_{B,2})$ being weakly prime or having edge representativity at least 4.
\end{proof}

\begin{rem} Lemma 3.9 of  ~\cite{ETS} has a fourth part stating that if distinct blue regions meeting a single crossing circle meet at the same crossing of $K_{B,2}$, then that crossing is associated with the crossing circle. This part
does not carry through to the generalization of weakly generalized alternating knots.
To prove part four the authors in ~\cite{ETS} rely on their Lemma 3.1 which states that, in the case of usual alternating knots,
$K_2$  $K_{B,2}$ are prime. In our case
$K_2$ or $K_{B,2}$ are not prime,
while this case can give us a simple closed curve that goes through both the crossing circle and the crossing, because it intersects our knot exactly 4 times. This curve it could very well be an essential curve, and thus tell us nothing about how the crossing relates to the crossing circle. 
\end{rem}

With $K_{B,2}$ defined in our general context, we need to introduce a new definition to generalize a diagram being blue twist reduced. As we already have a notion of weakly twist reduced, we will model our definition off of this.

\begin{defn}
A diagram of a link is \textit{weakly blue twist reduced} if every disk $D$ in $F$ with $\bndry D$ meeting $\pi(L)$ in exactly two crossings, with sides on the blue checkerboard surface, either bounds a string of red bigons, or there is a disk $D'$ in $F$ meeting the diagram at the same crossings that bounds a string of red bigons.
\end{defn}

Next, we will generalize Lemma 3.4 of ~\cite{ETS}, using this adapted definition instead of the usual blue twist reduced definition.
\begin{lem}
\label{lem:wBluTwistReduced}
$\pi(K_{B,2})$ is weakly blue twist reduced. 
\end{lem}

\begin{proof}
Suppose a disk $D$ in $F$ with $\bndry D$ meeting only the blue regions, and intersecting $\pi(K_{B,2})$ in exactly two places. We can isotope $\bndry D$ away from any crossing disk so that when we add back in crossings to get $\pi(K)$, $\bndry D$ remains disjoint from the red surface. Because $\pi(K)$ is weakly twist reduced, either $D$ bounds a string of bigons, or there is another disk $D'$ in $F - D$, with $\bndry D'$ meeting the diagram in the same two crossings, that bounds a string of bigons. If it is the later, then note that $\bndry D'$ also only meets the blue region. In either case, these bigons must all be red bigons. Without loss of generality, suppose $D$ contains the red bigons.

Because $\pi(K)$ is weakly twist reduced, all of the red bigons must come from the same twist region. When we remove some bigons to get back to $\pi(K_{B,2})$, we either don't touch any of the bigons inside $D$, in which case we are done, or we remove all but one or two crossings from the region. If this is the case, then we have to have removed all but two crossings - $D$ intersects two distinct crossings, so there must be at least two distinct crossings left in the region. But then $D$ contains a single red bigon, and we are done.
\end{proof}

\section{Colored Graphs and Twisted Surfaces}

In this section,  we will prove a series of technical lemmas that we will need for the proof of Theorem ~\ref{thm:essential}. The lemmas aim to analyze   how certain disks might intersect our twisted surfaces. The combinatorics of these intersections will be encoded by certain colorer graphs planar graphs. 
Before we can begin our analysis, we need to establish some notation and terminology.
  
\subsection{Blue and red graphs}

Recall that we need to prove that the immersed blue and red surfaces $S_{B,2}$ and  $S_{R,2}$ are $\pi_1$-injective and boundary $\pi_1$-injective in $Y\setminus K$.

Working with both $S_{B,2}$ and  $S_{R,2}$  simultaneously will be difficult. Instead, we will take advantage of the fact that the construction of one surface doesn't involve the other. That is, we could have constructed, say, $S_{B,2}$, the blue twisted surface, without ever having mentioned the red surface. The only part that would change is where we put the crossing circles---in order to not disturb the red surface, a crossing circle can only be added to a highly twisted region if it will intersect the blue surface. Thus in our proofs, we will often be working with the $S_{B,2}$ and $R$, where $R$ is the usual red checkerboard surface for $\pi(K)$.

Recall the map $f:S_{B,2}\to Y\setminus K$ from the statement of Theorem ~\ref{thm:essential}.
 If $f$  is not  $\pi_1-$injective, then we have a disk $\phi:D\to Y\setminus K$ such that $\phi|_{\bndry D} = f\circ l$, where $l$ is an essential loop on $S_{B,2}$.
 
\begin{defn}
Let $\Gamma_B = \phi^{-1}(f(S_{B,2}))$ the pre-image of $f(S_{B,2})$ under $\phi$. 
\end{defn}

By working with $\Gamma_B$ will study the intersection $S_{B,2}$ intersects the disk $\phi(D)$.
By  Lemma 2.1 of ~\cite{ETS}, which also applies in our situation without changes, we have the following: 
\begin{itemize}
\item 
$\Gamma_{B}$ consists of a graph and a collection of simple closed curves. 
\item The vertices of the graph are points mapped to crossing circles of $K$. 
\item Each interior vertex of the graph  has valence a multiple of $2n_j,$ where $2n_j$ is the number of crossings removed from the twist region of $\pi(K)$ that corresponds to the twist region associated to the relevant crossing circle.
Each boundary vertex has valence $n_j+1$.
\end{itemize}

While the arguments from ~\cite{ETS}
are given in there for $Y=S^3$ and $F=S^2$, as long as our $Y$ has no boundary components, many of the proofs require no modification. 
This is because, in many cases, the proofs are ``local,`' and only reference a small neighborhood of where we are working. In this paper, instead of reproving everything, we will instead focus on the cases where the same proof doesn't work, requiring us to do something different, often either a completely different proof, or showing that the cases that occur when $Y\neq S^3$ or $F\neq S^2$ cannot happen.

To continue, recall the links $L_B$, $L_{B,2}$, and $K_{B,2}$ that we introduced in Section 4 before the statement of Lemma ~\ref{lem:opagree}. We now have three surfaces to consider---$B_2$ and $S_{B,2}$ in $Y\setminus L_{B,2}$ and $Y\setminus K$, respectively; $R_2$, the red checkerboard surface for $K_{B,2}$; and the crossing disks bounded by the crossing circles of $L_{B,2}$, which we color green. This also gives us three graphs to look at.

\begin{defn}
 First, $\Gamma_B$ is the same as defined above. Next, if we remove the vertices of $\Gamma_B$ (which map to crossing circles) from the disk $D$, we get an embedding of the punctured disk
  $\phi':D'\to Y\setminus L_{B,2}$. Let $\Gamma_{BRG}$ denote the pull-back of the union of all three surfaces via $\phi'$.
  Finally, by ignoring the edges and vertices of $\Gamma_{BRG}$ coming from the green surface, we get $\Gamma_{BR}$. Note that $\Gamma_B\subset \Gamma_{BR}\subset \Gamma_{BRG}$.
\end{defn}

\begin{figure}
	\centering
		\def\svgwidth{\columnwidth}
		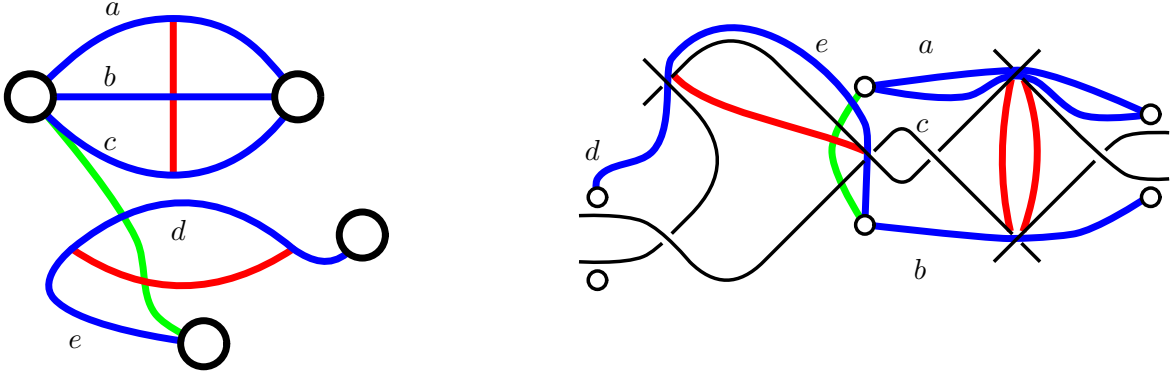

		\caption{\textit{(Left)} A possible configuration for $\Gamma_{BRG}$. \textit{(Right)} How $\Gamma_{BRG}$ would look when imposed on the diagram. The labeled blue edges in the graph map to the corresponding labeled arcs in the diagram.}
\end{figure}

We will be looking for what are called trivial bigons in  defined as follows in ~\cite{ETS}:

\begin{defn}
An edge of $\Gamma_{BRG}$ is \textit{trivial} if it is a blue arc disjoint from the red edges with endpoints distinct vertices of $\Gamma_B$ corresponding to the same crossing circle in $L_B$.
\end{defn}

An important lemma relating to trivial arcs is the following:

\begin{lem}[Lemma 4.11 ~\cite{ETS}]
If all but one of the edges of a region in $\Gamma_B$ are trivial, then the remaining edge is also trivial
\end{lem}

In particular, this means if one edge of a blue bigon is trivial, then the other one must be as well. This will allow us to divide families of adjacent bigons into two cases, trivial and non-trivial. Then, as shown in Lemma 2.6 of ~\cite{ETS}, $\Gamma_B$ must have more than $(N_{tw}/18)-1$ adjacent non-trivial bigons. So, if we can show that there must be less than $5\le (121/18)-1$, then we get a contradiction. As our only assumption was that $S_{B,2}$ was not essential, we will have that our twisted surface must be essential.

\begin{lem}[Lemma 2.6, ~\cite{ETS}]
\label{lem:minbig} 
The graph $\Gamma_B$ must have more than $(R_{tw}/18)-1$ adjacent non-trivial bigons, where $R_{tw}$ is the minimum number of crossings removed from a twist region.
\end{lem}
\begin{proof} The proof of Lemma 2..6 in ~\cite{ETS}
does not actually rely on either our knot or twisted surface, but is based purely on facts about the graph itself, and involves restricting to subdisks and subgraphs and removing trivial bigon families until we can get our result. As such, the proof and the lemma will continue to work  in our setting, as well.
\end{proof}

To continue recall that we must prove that $f:S_{B,2}\to Y\setminus \mathrm{int}(N(K))$ in the statement of Theorem ~\ref{thm:essential} 
is boundary $\pi_1$-injective: If not, then we can get a disk $\phi':D'\to Y\setminus \mathrm{int}(K)$, where $\bndry D'$ is a concatenation of two arcs, one mapped into $\bndry N(K)$ and the other essential in $S_{B,2}$ ~\cite[Lemma 2.2]{ETS}. 
\begin{defn} Define $\Gamma_B'$ to be the pull-back of  $f(S_{B,2}))$ via $\phi'$. That is,
 $\Gamma_B' = \phi'^{-1}(f(S_{B,2}))$.
\end{defn}

Once again the properties of  $\Gamma_B'$ remain the same if we replace $S^3\setminus K$ from~\cite{ETS} with
$Y\setminus K$. That is, the interior vertices of $\Gamma_B'$ have valence $2n_j$, while exterior vertices have valence $n_j+1$ (if they come from a crossing circle) or 1 (if the vertex maps to $\bndry N(K)$). 
In particular, the following lemma from ~\cite{ETS}  works also in our setting.

\begin{lem}[Lemma 2.7 ~\cite{ETS}]
\label{lem:mintri}
The graph $\Gamma_B'$ must have more than $(R_{tw}/18)-1$ adjacent non-trivial bigons or more than $(R_{tw}/18)-1$ adjacent triangles, where one edge of each triangle lies on $\phi'^{-1}(\bndry N(K))$.
\end{lem}

\begin{rem}
There are a couple of important things to note about the last two lemmas. First, while $R_{tw}$ is not necessarily equal to $N_{tw}$, we do know that $R_{tw}\ge 2\lceil N_{tw}/2 \rceil -2$. With $N_{tw}\ge121$, we get $R_{tw}\ge 120$, so there must be more than 5 adjacent non-trivial bigons in $\Gamma_B$. Second, our choice of $N_{tw}\ge 121$ is not optimal for these lemmas. In fact, the proof that $S_{B,2}$ is essential only needs $N_{tw}\ge 91$, so we have at least 5 adjacent bigons, or 5 adjacent triangles (we will prove this in the next section). The choice of 121 comes from a modification that will get us our final proof, Theorem ~\ref{thm:homarcs}.
\end{rem}

\subsection{Bigons in $\Gamma_{BRG}$} 

Here we will study the structure and combinatorial properties of the graphs $\Gamma_B$, $\Gamma_{BR}$ and $\Gamma_{BRG}$ defined in the previous subsection.
From the last subsection we know that these graphs
have a certain number of adjacent non-trivial (blue) bigons, we will be looking at how the disk must intersect other surfaces near the blue surface. Our goal here is to study the possible configurations of these bigons and successively rule out several of then. By the end of this section, we will be left with only one outcome---the family of bigons must intersect the red surface at least twice.

Recall that $\Gamma_B$ is the pull back of $S_{B,2}$ of a certain map
$\phi: D\to Y\setminus K$ and that  $\Gamma_{BRG}$ is the pull-back of $S_{B,2}$, as well as the red surface $R_2$ and the green surface, of a map $\phi: D' \to Y\setminus L_{B,2}$ and that $\Gamma_{BR}$
is obtained by  deleting the green edges of $\Gamma_{BRG}$. Below we summarize some basic facts about these graphs that we will be using:
\begin{itemize}

\item As outside of the twisted parts, the twist surfaces are embedded, the edges of $\Gamma_B$ (where the disk meets the blue surface) can only meet each other at the twisted components (where the crossing circles of $L_{B,2}$ intersected our blue surface). As such, the vertices of $\Gamma_B$ are exactly the points that map to crossing circles of $L_{B,2}$. This means that, if two blue edges are adjacent to each other on a vertex, it must be that they are on opposite sides of the crossing circle (as they must meet on the twisted part). 

\item In addition to what we have for $\Gamma_B$, we also know that red edges can't intersect red edges, nor can green edges intersect green edges (as both red and green surfaces are embedded). 

\item Red edges can meet green edges (at the crossing disk). Likewise, green edges can meet blue edges at vertices of $\Gamma_B$ (where the crossing disk meets $S_{B,2}$), and red edges can meet blue edges at crossings of the diagram. In addition, as $D'$ crosses the projection surface when it meets a blue or red edge, regions of $D'\setminus \Gamma_{BR}$ are mapped above or below the projection surface, switching from one to the other when they meet one of these edges.

\item Using a complexity-minimizing argument, where complexity is measured by the number of vertices in each graph, we get that there are no green edges, disjoint from blue, with both endpoints on red; and no green edges, disjoint from red, with both endpoints on a blue edge and one endpoint not a vertex. In addition, green edges cannot have both of their endpoints on the same blue vertex. In all three of these cases, we can use the arcs in question, in addition to the crossing disk, to find a homotopy that reduces how many (non-blue) vertices are in our graphs. So as long as we assume our disk's image, $\phi(D')$, and the induced graphs are minimal, these cannot happen.
\end{itemize}

As before, the lemmas   are generalizations of the case when $F=S^2$ and $Y=S^3$. Most of the generalizations are relatively minor, with only a few complications brought in by the fact that we aren't working on a sphere anymore. For completeness sake, we will also talk about the proofs that still hold, although in much broader and looser terms. For the complete picture, refer to ~\cite{ETS}.

The configuration---a blue-red bigon---needs a separate proof. We deal with this as follows:

\begin{lem}
\label{lem:BlueRedBigon}
The graph $\Gamma_{BR}$ has no bigons with one blue side and one red, whether or not the bigon meets the green surface.
\end{lem}

\begin{proof}
If there is a bigon with only one red side and one blue side, then it must be mapped completely above or completely below the projection surface. Without loss of generality, we may assume it is mapped above. Then the union of the two sides of our bigon will form a simple closed curve $\gamma$ meeting $\pi(K_{B,2})$ exactly twice with a crossing on either side. If this curve bounds a disk, we contradict $\pi(K_{B,2})$ being weakly prime. Because we are not working on the sphere, though, we are not guaranteed this. Instead, we must use the fact that the edge representativity (see Definition ~\ref{def:rep}) of $\pi(K_{B,2})$ is at least four. If $\gamma$ does not bound a disk, then we have an essential curve intersecting our knot transversely exactly twice, a contradiction of ~\ref{lem:edgerep}. So then $\gamma$ must bound a disk with a crossing on either side, and so we get our desired contradiction.
\end{proof}

Putting our observations above together, we get two important results. First, our disk $D$ only meets the blue surface in ``interesting'' places---that is, near or through the crossing disks. Second, there are no monogons in $\Gamma_B$: any that existed couldn't meet either the red or green surfaces, as mentioned above, and so would have to connect two opposite sides of the crossing circle. However, by Lemma ~\ref{lem:opagree}(3), this can't happen. While this later is not used immediately, we will use it later on, and so state it as a lemma:

\begin{lem}[ Lemma 4.7 ~\cite{ETS}]
\label{lem:nomono}
The graph $\Gamma_B$ has no monogons.
\end{lem}
 
As we want to work eventually with non-trivial bigons, it's important that we understand how trivial arcs and bigons work. As stated along with the definition, if one edge of a bigon is trivial, so is the other, giving us our trivial bigon families. In addition to this, we also find that trivial arcs have some restrictions to them---they must have at least one endpoint on the boundary of our disk, but cannot be a subset of the boundary. With that, we will set this aside to start looking at non-trivial bigons. 

\begin{lem}
\label{lem:BBigIntRorG}
In the graph $\Gamma_{BRG}$, there are no non-trivial bigons with two blue sides, disjoint from red and green edges.
\end{lem}

\begin{proof}
The proof is essentially the same as that of Lemma 4.12 in ~\cite{ETS}: Suppose we had such a non-trivial bigon. Because the bigon is non-trivial, the two vertices must correspond to separate crossing circles. When we put back in the crossings to get $\pi(K)$, we get a simple closed curve $\gamma$ from the bigon. The blue edges of the bigon remain disjoint from any crossings, and connect across where the crossing disks met the projection surface. In particular, $\gamma$ intersects $\pi(K)$ in exactly two crossings. We want to show that $\gamma$ bounds a disk in our projection surface $F$---this will allow us to appeal to $\pi(K)$ being twist reduced and will give us a contradiction. So suppose $\gamma$ doesn't bound a disk in $F$. Then the curve must be essential.

If $\gamma$ bounds a compression disk, and $r(\pi(K),F)>4$, then, after a small isotopy to take $\gamma$ away from crossings, we get a curve bounding a compression disk that intersects our diagram only four times, a contradiction. So now we need to see what happens if $\gamma$ doesn't bound a compression disk.

As $\gamma$ bounds a bigon disjoint from red and green edges, it must bound a disk in a region in the pre-image of the blue surface in our disk. By our assumptions, $\gamma$ can't bound a disk in $F$, nor can it bound a compressing disk for $F$. Therefore, the region it bounds must pass through our projection surface. But as our bigon is disjoint from red, this cannot happen (regions of the disk can only pass through the projection surface at blue or red edges, neither of which can be in our bigon). So then we get a contradiction of $\gamma$ not bounding a compression disk.

Putting it all together, we see that we have a simple closed curve $\gamma$ which must bound a disk on $F$. But then, as $\gamma$ meets $\pi(K)$ in exactly two crossings and bounds a disk, because $\pi(K)$ is weakly twist reduced, we must have both of the crossings correspond to the same twist region, a contradiction of the bigon being non-trivial.
\end{proof}

So blue bigons must have some other colored edge intersecting them. A bigon intersected by another edge gives us two triangles, so it is natural to consider how adjacent triangles might act. As red cannot intersect red and green cannot intersect green, we find ourselves with four possible adjacent triangles, constructed as follows:
\begin{enumerate}
	\item Blue-Red bigon intersected by green (this can't happen, as there are no blue-red bigons);
	\item Red-Green bigon intersected by blue (we will prove this can't happen below);
	\item Blue bigon intersected by green (this can't happen, as we can't have a blue-blue-green triangle by minimality, by ~\cite{ETS} Lemma 4.16);
	\item Blue bigon intersected by red.
\end{enumerate}

\begin{figure}[!htb]
	\centering
	\def\svgwidth{\columnwidth}
	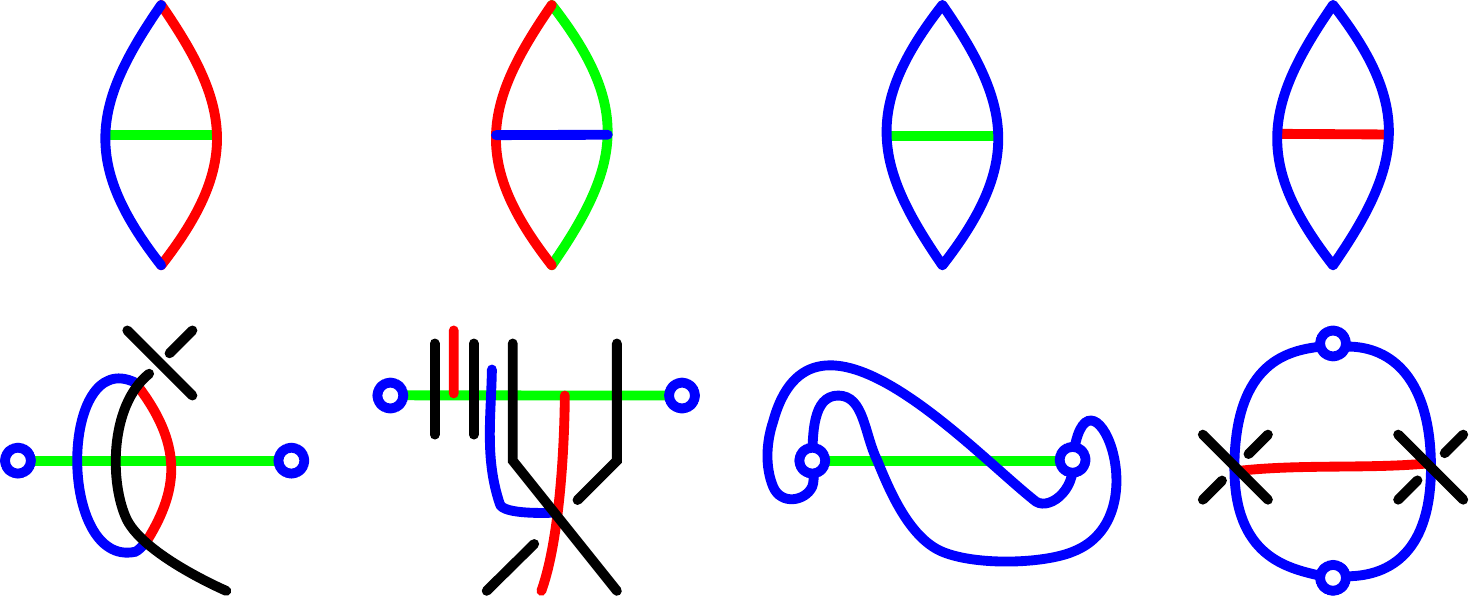
	\caption{\textit{(Top)} The four different types of adjacent triangles coming from a bigon. \textit{(Bottom)} How the bigons might look in the diagram. Note that, in the red-green bigon case, the two red edges in the diagram are, in fact, the same edge, and must ``wrap around'' the diagram to create the bigon.}
\end{figure}

\begin{lem}
In $\Gamma_{BRG}$, there are no pairs of red-green-blue triangles adjacent across a blue edge. More generally, no pair of green edges can be added to $\Gamma_{BR}$ in such a way that the result is a pair  of triangles adjacent across a blue edge.
\end{lem}

\begin{proof}
We will prove the second statement. If such an edge could be added, we can find the innermost pair of triangular regions, so there are no additional red or blue edges in the triangles. Then one of these triangles must be mapped above the projection plane, and the other below. As there is only one place where red meets blue, there is only one crossing in the image of this diagram, so both red edges must meet this crossing, and meet no other crossing in our diagram.

Looking at the green edges, we see that both meet the same blue edge at the same point. This means that these green edges are related to the same crossing circle. As our two red edges meet these green edges, our red edges must both run back to the same crossing circle. So, we can create a simple closed curve with crossings on either side by following along the red edges, jumping from one to the next at either the crossing or the crossing circle. This curve intersects our knot exactly twice, so by the representativity and edge representativity of our knot, must bound a disk. However, this contradicts weak primality, so this cannot happen. 
\end{proof}

The other two triangles we need to care about are blue-blue-green triangles (the third case listed above), and blue-blue-red triangles (the fourth case). In fact, with no change in proof from ~\cite[Lemma 4.12]{ETS}, there cannot be any blue-blue-green triangles in $\Gamma_{BRG}$. The reason we can do this here is because we don't have to worry about any of the curves we create from the blue and green edges not bounding a disk on $F$---otherwise we would have a curve intersecting our knot only twice and not bounding a disk. We will focus, then, on blue-blue-red triangles.

\subsection{Blue-Blue-Red Triangles}

Our ultimate goal is to show that there can only be so many adjacent blue bigons in $\Gamma_B$, contradicting Lemma ~\ref{lem:minbig} that says we need at least five. As a red edge must intersect any family of non-trivial blue bigons, we can look at blue-blue-red triangles constructed from the bigons and the red edge intersecting them.

One of the ways we will go about generalizing the results of ~\cite{ETS} is by using a sort of ``local'' property---most of the proofs take place in a small subsection of the diagram, and so we don't necessarily need to know what's happening beyond that for a given proof. To take advantage of this for generalized surfaces, however, we will need to show that the subsections we care about are similar enough to subsections on a sphere. Namely, we will need to show that these small parts don't wrap around the surface, and instead lie on a disk. Once we have this, most of the proofs follow naturally.

First up, one of the more common cases we will need to deal with is where we have a triangle with two blue edges and a red edge. That such triangles induce a disk on $F$ is split into two parts below, depending on how the red edge might pass through the crossing circle of the vertex. One important thing to note that will come up several times is that we don't actually need our triangle, or any of our future shapes, to bound a disk. In fact, looking at the case when the red edge passes through the crossing circle, this would be impossible. Instead, we just want it to lie within a disk, and hence induce a disk on $F$.

\begin{lem}
A blue-blue-red triangle in $\Gamma_{BR}$ which, in the image on $F$, has none of its edges passing through the crossing circle of the blue vertex induces a disk on the projection surface $F$.
\end{lem}

\begin{figure}[!htb]
	\centering
		\def\svgwidth{.5\columnwidth}
		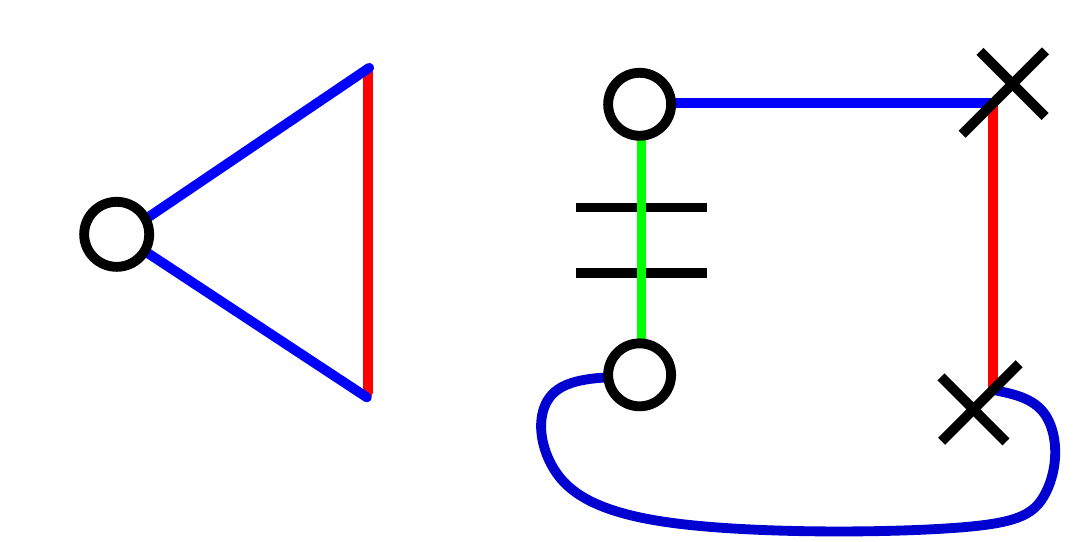
\caption{\textit{(Left)} A blue-blue-red triangle as viewed on the graph $\Gamma_{BR}$. \textit{(Right)} The same triangle as viewed on the knot diagram, $\pi(K_{B,2})$.}
\end{figure}

\begin{proof}
First, note that the triangle gives us a simple closed curve on $F$: let $\gamma$ be the curve that follows our triangle along the edges, and joins the blue edges across the crossing disk associated to the vertex where the blue edges meet. By Lemma 4.18 in ~\cite{ETS}, the triangle meets two distinct crossings, and so then $\gamma$ must also meet two distinct crossings. We want to show that $\gamma$ bounds a disk on $F$. If it doesn't, there are two possibilities we will consider.

The first case: $\gamma$ could bound a compression disk. If it does, then note when we put crossings back in, that $\gamma$ intersects our knot in exactly four places: once across each of the crossings, and twice across the crossing disk of the vertex. But then we have a curve bounding a compression disk meeting our knot four times, contradicting the assumption that $r(\pi(K), F)>4$.

So then if $\gamma$ is essential and doesn't bound a compression disk, we can look at the triangular region $\gamma$ bounds in the disk. Because $\gamma$ bounds neither a compression disk nor a disk on $F$, the triangular region must pass through the projection surface. But then our region must have a red or blue edge in it, as it can only pass through the projection surface through an edge of $\Gamma_{BR}$. Then we don't have a blue-blue-red triangle, and so we are done.
\end{proof}

\begin{lem}
A blue-blue-red triangle in $\Gamma_{BR}$ with the image of the red edge passing through the crossing circle of the vertex induces a disk on the projection surface $F$.
\end{lem}

\begin{proof}
As above, we know that the triangle has to meet two distinct crossings. But now, as our red edge must cross through the green edge, we are not guaranteed a single simple closed curve. Instead, we shall construct two separate disk , and then glue them together. Call our blue edges $b$ and $c$ and our red edge $j$. Let $\gamma_1$ be the curve that follows $b$, then $j$ until we get to the crossing circle, and then the green edge back to $b$. Likewise, let $\gamma_2$ be the curve that follows $c$, then $j$ until we get to the crossing circle, and then jumps back to $c$ along the green edge.

First, each of these curves must intersect our diagram exactly twice - once at the crossings, and another time through the crossing disk. This allows us to immediately say that both $\gamma_1$ and $\gamma_2$ bound disks, as $e(\pi(K), F)\ge 4$. Next, note that $\gamma_1$ intersects $\gamma_2$ at exactly one point---where red meets green. This is because a red edge cannot intersect itself, and blue edges can only meet at vertices. But, as blue edges adjacent to the same vertex, they must map to different sides of the crossing region, and so are disjoint. Thus the only point they can meet is the one they share---where the crossing circle meets the red edge.

Now, we can finally construct our entire disk. Let $D$ be the union of the disk bounded by $\gamma_1$, the disk bounded by $\gamma_2$, and a small neighborhood of the intersection of the crossing circle and red edge. As $\gamma_1$ doesn't interfere with $\gamma_2$ outside of this neighborhood, we do in fact have a disk as opposed to an annulus, and so are done.
\end{proof}

We can use these disks to construct even larger disks, with some caveats---we need to be careful that when we glue disks together, we aren't creating an annulus.

\begin{lem}\label{lem:bbrtrng}
Families of blue-blue-red triangles, adjacent along blue edges, bound a disk.
\end{lem}

\begin{proof}
\begin{figure}[!htb]
	\centering
		\def\svgwidth{.75\columnwidth}
		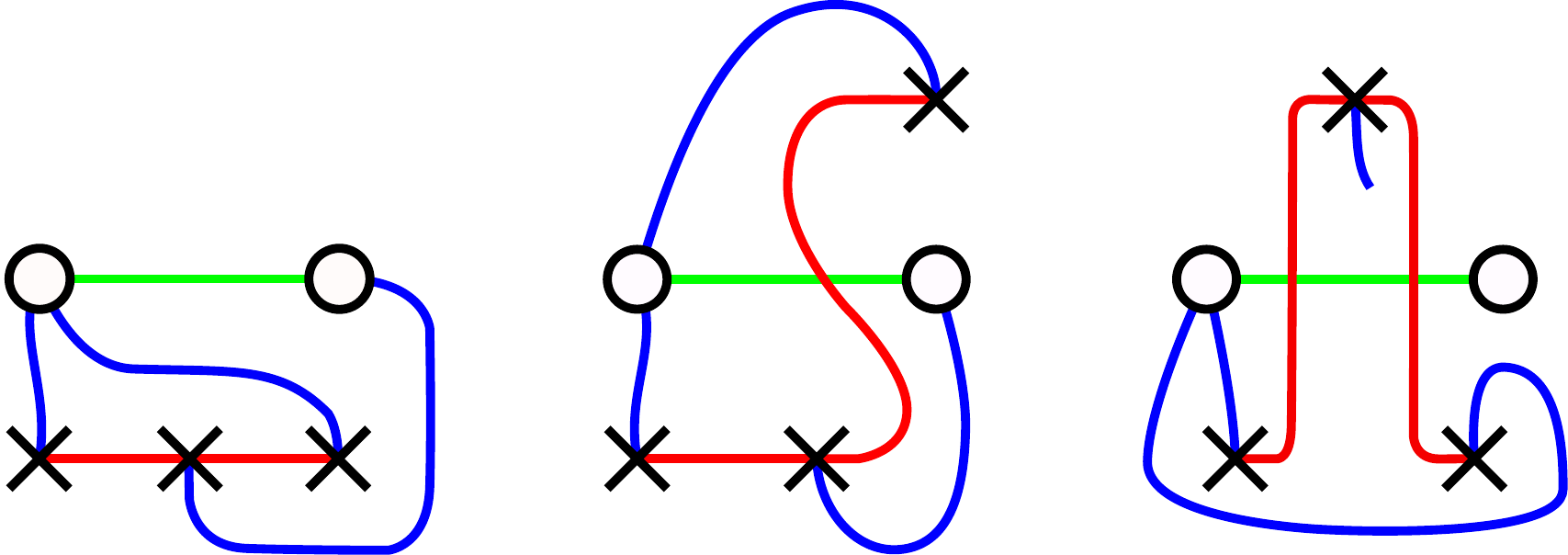
	\caption{The three possible ways two blue-blue-red triangles can have their images glued together.}
		\label{fig:twotri}
\end{figure}
If we have a single blue-blue-red triangle, we are done by above. So we will suppose we have a group of them, and consider what happens when we want to add one more. As blue edges can only intersect at vertices, and red edges can never intersect, we only have to worry about two triangles---the one we are adding and the triangle that shares an edge with it. There are three cases we have to worry about, as shown in Figure ~\ref{fig:twotri}. First, both triangles have their red edges not passing through the crossing circle. Then both triangles bound a disk, and share a single continuous segment---their common blue edge and the green edge---so we are done.

Second, one of the triangles doesn't pass through the crossing circle, but the other does. In this case, for the triangle that does pass through the crossing circle, we get two separate disks, say $D_1$ and $D_2$. The other triangle gives us a third disk, $D_3$. We will union these all together in the right order to get what we want. First, note that $D_1$ and $D_3$ agree only on a portion of the green edge, so gluing them together will still give us a disk. Then the boundary of $D_2$ only agrees with the boundary of this new disk in one continuous segment---the remainder of the green edge and the common blue edge. So we can add $D_2$ in and get our whole disk.

Finally, we could have both red edges pass through the crossing disk. Only one red region can pass through the a crossing disk, by definition. So for both red edges to pass through the crossing disk, they must be in the same region. But then at least one crossing, the one that meets both red edges, has opposite (red) sides agreeing, contradicting Lemma ~\ref{lem:opagree}. So this case cannot happen, and we are done.

\end{proof}

\begin{lem}\label{lem:bluerec}
A blue rectangle in $\Gamma_{BR}$ with two vertices that correspond to the same crossing disk and a red diagonal has both crossings associated to the crossing circle of the vertices of the rectangle.
\end{lem}

\begin{proof}
\begin{figure}[!htb]
	\centering
		\def\svgwidth{.75\columnwidth}
		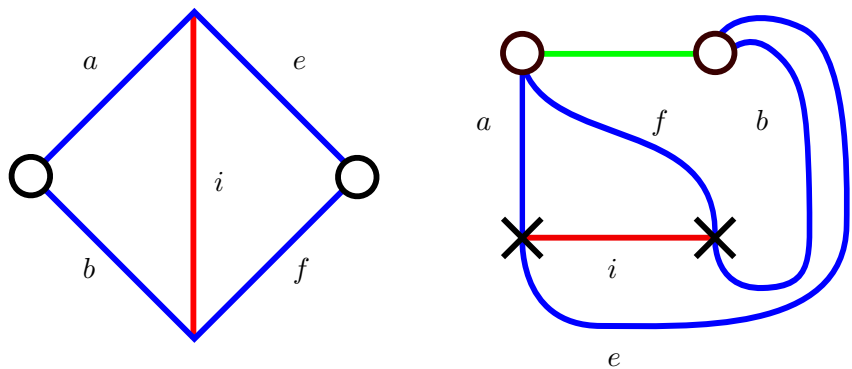

	\caption{\textit{(Left)} A blue rectangle with a red diagonal in the graph $\Gamma_{BR}$. \textit{(Right)} A blue rectangle as viewed on the diagram, where both vertices represent the same crossing circle.}
	\label{fig:bluerec}
	
\end{figure}
Our image has to look like one of the above---the red edge tells us we need at least two crossings. In either case, look at the two curves $a$-$i$-$f$ and $b$-$i$-$e$. Each of these curves intersects our knot transversely exactly twice. This means, then, as both the representativity and the edge representativity (defined in Definition ~\ref{def:rep}) are at least four, each of these curves must bound a disk. Now, though, we have two disks who agree at only one segment of the boundary. As such, we can glue the disks together without risking an annulus, and thus giving us another disk, this one with boundary $a$-$f$-$e$-$b$, and so we are done.
\end{proof}

To make sure we are on the right track, it is important to look at why we are studying these blue-blue-red triangles. We know that we have to have so many adjacent non-trivial blue bigons. As it turns out, given these blue bigons, we will eventually show that a red edge must pass through these bigons. If only one red passes through, we get a two collections of blue-blue-red triangles, with the families adjacent across the red edges. So it's important to see what happens here.

Before we begin studying these in earnest, there is an additional lemma that we will need here. In summary, Lemma 4.17 from ~\cite{ETS} tells us that, if we have three adjacent blue-blue-red triangles that are adjacent at the vertex of the blue edges, two properties must hold. First, no green edge can intersect an interior blue edge (so, in figure ~\ref{fig:adjtri}, edges $b$ or $c$). Second, if a green edge meets the interior red edge (so edge $j$), the green edge must run to the common vertex of the triangles. As we have families of blue-blue-red triangles bounding a disk by Lemma ~\ref{lem:bbrtrng}, the exact same proof as in ~\cite{ETS} works. With this in mind, we can move on to looking at our triangles.

\begin{lem}
\label{lem:adjtri}
The graph $\Gamma_{BR}$ cannot contain two pairs of three adjacent blue-blue-red triangles adjacent across the red edges.
\end{lem}

\begin{proof}
Suppose it did. Then we look at such a group of triangles, as pictured in Figure ~\ref{fig:adjtri} (left). First, we need to show that the two vertices of the triangle group can't agree (that is, can't be related to the same crossing circle). If the vertices do agree, then the blue edges $a$, $f$, and $c$ must be in the same region, and the blue edges $e$, $b$, and $g$ must be in the same region.
\begin{figure}[!htb]
	\centering
		\def\svgwidth{.75\columnwidth}
		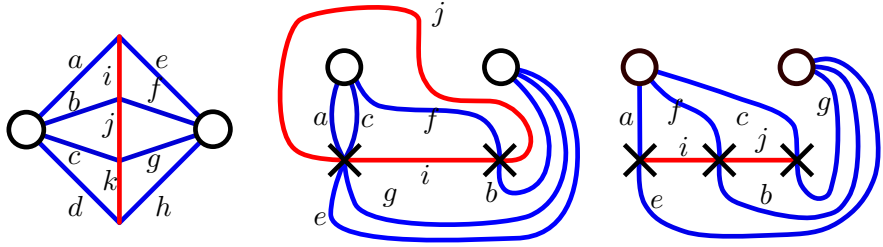

		\caption{\textit{(Left)} Two pairs of three adjacent blue-blue-red triangles. \textit{(Center and Right)} The two possibilities if the vertices map to the same crossing circle, one with two crossings and one with three.}
		\label{fig:adjtri}
\end{figure}
From there, we want to show that the endpoints at each of the edges are all associated with the crossing circle of the vertices. First, look at the $a$-$e$ crossing and the $b$-$f$ crossing. As $a$-$e$-$f$-$b$ forms a blue rectangle with two vertices and a red edge down the diagonal, we get a disk with boundary blue-blue-green. We can assume the boundary is $a$-$e$-green. But then, by definition, the crossing at $a$-$e$ must be associated with the crossing disk. Then, by two applications of Lemma ~\ref{lem:bluerec} (one with $a$-$e$-$f$-$b$ and one with $b$-$f$-$c$-$g$), we get that the crossings of $b$-$f$ and of $c$-$g$ must be associated to the crossing circle as well.\\

As we are working with $L_{B,2}$, this does not automatically give us a contradiction---the diagram can still work if we have only two distinct crossings, such as in Figure ~\ref{fig:adjtri}. However, suppose there are only two distinct crossings for our triangle. Then we can get several curves that bound disks in $F$ that intersect our diagram exactly twice: $c\cup i\cup f$, $e\cup i \cup b$, $f \cup a \cup j$, and $b\cup j\cup g$, as follows.
To get the last curve, look at the rectangle $b$-$f$-$c$-$g$. It's disk must contain the $j$ edge, so take the sub disk bounded by $b\cup j\cup g$. Then, for the curve $f\cup a \cup j$, note that $a$ must be parallel to $c$, so use the sub disk bounded by $f\cup j\cup c$, and replace $c$ by $a$. In similar fashion, but working with the rectangle $a$-$e$-$f$-$b$ and the edge $i$, we can get the other two curves. But then, as $\pi(K_{B,2})$ is weakly prime (see Definition ~\ref{def:wprime}, we get that there are only two crossings on $F$.\\

We now have two possible options. First, $\pi(K_{B,2})$ can't be a single component knot, as any attempt to connect two crossings will result in either the unknot (which can't happen, as $e(\pi(K_{B,2}),F)\ge 4$) or a non-alternating knot on $F$. On the other hand, if $\pi(K_{B,2})$ is a link, by following the link around the outside, we can find a curve that completely encompasses the link. As $e(\pi(K_{B,2}),F)\ge 4$, this curve must bound a disk on $F$. But then we can find an essential curve that does not intersect our knot diagram at all, a contradiction. So there cannot be only two crossings for $\pi(K_{B,2})$ on $F$, and so the two vertices of our group of blue-blue-red triangles cannot agree.\\

Now we want to see how the red edges might intersect the crossing disks. First, the red edge $j$ cannot intersect either crossing disk by Lemma 4.17 from ~\cite{ETS}, as if it did, we would have to have a green edge intersecting either $b$ or $c$, contradicting the lemma. Now, looking at $i$, we'll see that it can only intersect one of our crossing disks. Otherwise, if the crossing disk corresponding to $e$, $f$, and $g$ intersects $i$, then we must have a green edge running from $i$ to $a$. This means that $f$ and $a$ are in the same region. On the other hand, if the crossing disk corresponding to $a$, $b$, and $c$ intersects $i$, we have a green edge running from $i$ to $e$, telling us that $b$ and $e$ are in the same region. Then we can create two curves, one from $a$, $b$, and green, and the other from $e$, $f$, and green. Each intersects our knot transversely at the $a$-$e$ crossing, but also intersects our knot at either the green edge corresponding to their respective crossing circles. But then we have a single crossing associated to two distinct crossing circles, which cannot happen. So then $i$ can only intersect, at most, one of the crossing disks. Assume, then, it doesn't intersect the crossing disk associated to $a$, $b$, and $c$.\\

Next, we need to look at the crossings at the endpoints of $a$ and $e$, which we'll call $x$, and at the endpoints of $c$ and $g$, which we'll call $y$. We know that $x$ must be distinct from the crossing at the endpoints of $b$ and $f$, and $y$ must be distinct from this third crossing as well. So we now want to show that $x$ and $y$ are distinct from each other. If they are not, then look at the curve $i\cup j$. This will give us a closed curve meeting both crossing disks, a contradiction to the crossing disk at $a$, $b$, and $c$ not meeting either $i$ or $j$.\\

\begin{figure}[!htb]
	\centering
	\def\svgwidth{.75\columnwidth}
	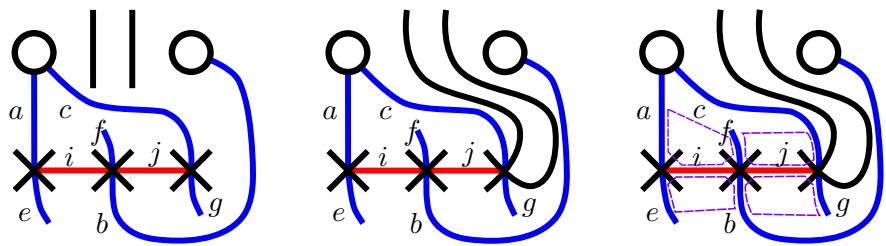
	\caption{\textit{(Left)} Focusing on the crossing circle of $a$, $b$, and $c$, our diagram must look something like this. \textit{(Center)} Once we know that $y$ is related to the crossing circle, we can draw in some extra information about our knot. \textit{(Right)} The dotted lines are disk-bounding curves that will allow us to show $x$ and $z$ are also related to the crossing circle.}
	\label{fig:adjtri2}

\end{figure}

We want to show that this can't happen by showing the twist region associated to the shown crossing circle has three crossings---namely, the three drawn crossings. Note that edges $g$ and $e$ must lie in the same blue region, as they share a vertex (the crossing circle of $e$, $f$, and $g$) and are mapped to the same side of that vertex, but also $g$ is bounded by $b\cup j \cup c\cup$ the green edge of the crossing disk, while $e$ lies outside of it, so we should figure out how they meet. As $g$ and $c$ are at opposite sides of a crossing, by Lemma ~\ref{lem:opagree}, $g$ and $e$ are not in the same region as $c$, so they can't cross over $c$. Because $j$ is red, neither $g$ nor $e$ can meet it in the graph. So the only remaining options are that one of $g$ or $e$ meets the blue edge $b$ or the crossing circle associated to $a$, $b$, and $c$. In either case, this will give us that $g$ and $e$ are in the same region as $b$.

Now we will show that the crossings $x$ (at $a$-$e$), $y$ (at $c$-$g$), and $z$ (at $h$-$f$) are all associated to the same crossing circle. As we have already shown that each of these crossings are distinct, we will get our desired contradiction. First, we show that $y$ is associated with the crossing circle of $a$, $b$, and $c$. To do this, we create a curve $\gamma$ starting at $y$, following along $c$, using the crossing disk's green edge to get to $b$, and then switching over from $b$ to $g$ to get back to $y$. Because $g$ is bounded by $b\cup j\cup c\cup$ the green edge, we can take $\gamma$ to lie inside the disk on $F$ bound by the triangle $b\cup j\cup c$ so that $\gamma$ too bounds a disk on $F$. But now, when we put crossings back in to get $K$, we have a disk on $F$ whose boundary intersects our knot at exactly two crossings, and so both must be related to the same crossing circle - namely, the crossing circle of $a$, $b$, and $c$. So then $y$ is associated to this crossing circle.\\

In order to work with the crossings $x$ and $z$, we will need to know a bit more about $f$---namely, that it is in the same region as $a$ and $c$. Clearly, if $f$ crosses $a$ or $c$, then they must be in the same region as $f$. So then let's suppose it doesn't. Then note that the edge $f$ is bounded by $a\cup c\cup j\cup i$. As $f$ is blue, it can't cross the red edges $i$ or $j$. Likewise, neither $e$ nor $g$ can cross $i$ or $j$. So, in order for these three blue edges to form the triangles in our graph, the crossing disk of $e$, $f$, and $g$ must enclose either $i$ or $j$. By the same reason the crossing disk of $a$, $b$, and $c$ can't meet $j$, the crossing disk of $e$, $f$, and $g$ can't meet $j$ either. So then the crossing disk must intersect $i$. But then, as shown in Lemma 4.17 of ~\cite{ETS}, we must have the green edge of the crossing circle intersect $a$. Now, with the green edge how it is, we must have $f$ be in the same region as $a$ and $c$.\\

Our last step is to show that $x$, $y$, and $z$ connect to form bigons, thus are three distinct crossings in $\pi(K_{B,2})$ with the associated to the same crossing circle, a contradiction. To do this, note that we can form four closed curves intersecting our knot exactly twice: $i \cup \frac{1}{2}b\cup \frac{1}{2} e$; $i\cup \frac{1}{2} f\cup \frac{1}{2} a$; $j\cup \frac{1}{2}g\cup\frac{1}{2}b$; and $j\cup \frac{1}{2}c\cup\frac{1}{2}f$, as seen in Figure ~\ref{fig:adjtri2} (right). As we are assuming that $e(\pi(K_{B,2}), F)\ge 4$, each of these curves must bound a disk. But then, because $\pi(K_{B,2})$ is a weakly prime diagram, and because there are crossings outside of these disks, the intersection of the disks and $\pi(K_{B,2})$ must be a single embedded arc. But then, drawing the result, we see that these three crossings form bigons, and so must be associated to the same crossing circle. As we are working with $K_{B,2}$, and at most two crossings can be associated with a crossing circle, we're done.
\end{proof}

As mentioned earlier, our want to study adjacent families of blue-blue-red triangles came from a red edge intersecting blue bigons. As this can't happen, we can't have just one red edge intersect our bigons. So, the next case to consider is two or more red edges intersecting the bigons.

\section{Triangles and Squares}
To review where we are, we know that any non-trivial blue bigon families we have must be intersected by at least two (parallel) red edges. We can then break these bigons up into blue-blue-red triangles and red-blue-red-blue rectangles. The triangles come from the ends of the bigons, and must meet the vertex of the bigon, while the squares are from the interior, and have a red edge on either end. In this section, we will focus on adjacent triangles and squares. Our goal now is to show that we can only have so many of these triangle-square pairs stacked on each other:

\begin{restatable}{lem}{fivepairs}
\label{lem:fivepairs}
There cannot be five adjacent triangle-square pairs for the diagram $\pi(L_{B,2})$.
\end{restatable}

Once we have proved this, we can combine it with Lemma ~\ref{lem:minbig}, which tells us we must have at least five adjacent non-trivial blue bigons, to get a contradiction and show that $S_{B,2}$ is essential. To do this, we will first prove that these triangle-square pairs have to sit inside a disk on $F$, and hence can't have any topology to them. Next, we prove ~\cite[Lemma 3.2(4)]{ETS} in a few specific cases, showing that if two blue sides of a crossing meet the same crossing circle, then that crossing must be associated to that crossing circle. With these, we can then use the proof of ~\cite[Section 5]{ETS}, with little modifications to prove Lemma ~\ref{lem:fivepairs}.

As before, where, when working with triangles, we saw that graphs induced disks on the projection surface, we will want to do something similar when working with triangle-square pairs, that is, when we glue one edge of a triangle to an edge of a square

\begin{lem}\label{lem:brrect}
A blue-red rectangle with interior disjoint from the vertices of $\Gamma_B$ induces a disk on the projection surface $F$. Furthermore, a triangle-square pair, with interior disjoint from the vertices of $\Gamma_B$ induces a disk on the projection surface $F$.
\end{lem}

\begin{proof}
We will prove this by starting with the simplest case, a blue-red rectangle with interior disjoint from blue and red edges, and expand from there. So assume our rectangle has interior disjoint from blue and red edges. Then, as we have done previously, we can form a simple closed curve $\gamma$ by taking the boundary of the rectangle, as below.

\begin{figure}[!htb]
	\centering
	\def\svgwidth{.75\columnwidth}
	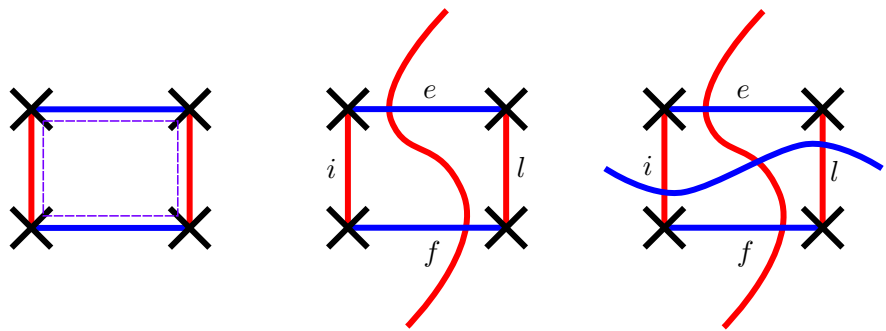
	\caption{\textit{(Left)} A blue-red rectangle disjoint from all other edges. The purple curve represents our $\gamma$. \textit{(Center)} A blue-red rectangle that might intersect a red edge, but is disjoint from blue edges. \textit{(Right)} A blue-red rectangle that might intersect any number of blue or red edges.}
\end{figure}

As red edges can't intersect red edges, neither of our two red edges can cross any additional crossings. Likewise, as blue edges can only meet blue edges at vertices, neither of our blue edges can meet additional crossings. So then $\gamma$ intersects our knot exactly four times. As $r(\pi(K), F)$ is strictly greater than four, $\gamma$ cannot bound a compression disk of $F$. On the other hand, if $\gamma$ is an essential curve, the disk it bounds must pass through the projection surface $F$. But such a disk can only pass through $F$ at a red or blue edge, which our rectangle is disjoint from. So then $\gamma$ must bound a disk on $F$.

Next, suppose our rectangle has interior disjoint from blue edges, but not necessarily from red edges. We first want to show that our red edges must run entirely through the rectangle, without any crossings in the interior. First, if any of the red edges meets a crossing, it must then either meet a blue edge at that crossing or meet no other edges at that crossing (as red edges can't meet red edges, and can only meet green at crossing circles). If it meets a blue edge, as our interior is disjoint from blue edges, that means the crossing must be on the boundary blue edge, and we are good. If it meets no other edges, then our crossing is in the interior. But then, we can go from one side of the red edge to the other by going around the crossing, with this path never intersecting our graph. But then, as the projection of the disk switches from above to below the projection plane at blue and red edges, we have a single region of our disk that must be both above and below this plane, which cannot happen. So any crossings our red edge meets must be on the boundary, and we may proceed.

Now we will proceed by induction. We already know how to proceed if we have no red edges. Now suppose we can get a disk if our rectangle has at most $n-1$ red edges in the interior. Label the blue boundary edges $e,f$ and the red boundary edges $i,l$. Given a rectangle with $n$ red edges in the interior, pick one of them and label it $\alpha$. As red edges cannot intersect red edges, this red edge must be parallel to all of the other red edges, including those on the boundary. Then we can come up with two smaller blue-red rectangles, each with less than $n$ red edges in the interior, by splitting along $\alpha$. The boundaries of these rectangles will be $\gamma_1 = i\cup \frac{1}{2} e\cup \alpha\cup \frac{1}{2}f$ and $\gamma_2 = \alpha\cup\frac{1}{2} e\cup l\cup \frac{1}{2} f$. Each of these curves, by induction, will give us disks on $F$. To get a disk on $F$ bounded by $i\cup e\cup l\cup f$, we can glue our two smaller disks along their common edge $\alpha$, and we are done with this case.

Finally, we deal with the broad case of a rectangle that is only disjoint from the vertices of $\Gamma_B$. As $\Gamma_B$ is a connected graph, and blue edges only meet blue edges at vertices/crossing circles, our blue edges must run all the way from one boundary (red) edge to the other. As it turns out, this doesn't actually matter for our proof - we just need the blue edges to run from one red edge to the next. But if the blue edges didn't run all the way from one end to the other, we could find a region where our disk was both above and below the projection plane, a contradiction.

Once again, we proceed by induction. If our rectangle has no blue edges in the interior, we can use what we've shown above to find a disk. Now suppose we can get a disk if our rectangle has at most $n-1$ blue edges, and suppose our rectangle has $n$ blue interior edges. Pick a blue edge and call it $\beta$. As blue edges can only meet blue edges at vertices, $\beta$ must run parallel to both the boundary edges $e$ and $f$, as well as the other interior blue edges. So we can construct two smaller blue-red rectangles with boundary $\gamma_1 = e\cup \frac{1}{2}l \cup \beta \cup \frac{1}{2}i$ and $\gamma_2 = \beta\cup \frac{1}{2}l\cup f\cup\frac{1}{2} i$, with each rectangle having less than $n$ blue edges. By induction, we can then find disks on $F$ with boundaries $\gamma_1$ and $\gamma_2$. To get a disk with boundary $e\cup i \cup f\cup l$, we can glue the disks for $\gamma_1$ and $\gamma_2$ along $\beta$, and we are done.

Finally, we will show that a triangle-square pair with no interior $\Gamma_B$ vertices induces a disk on $F$. We can split the triangle-square pair into a blue-blue-red triangle and a blue-red rectangle, whose boundaries share a red edge. The triangle bounds a disk by ~\ref{lem:bbrtrng}, while the square bounds a disk by above. So we can form a disk for the triangle-square pair by gluing these two disks together along their common red side.
\end{proof}

Now we will examine what three adjacent triangle-square pairs might look like. As we now know that the triangle-square pairs must all lie on a disc, and thus there is no genus involved, we don't have to worry about possibilities beyond the four drawn in Figure 17 of the paper. Likewise, many of the lemmas proven for $F=S^2$ will also hold for a general projection surface. As usual, we only have to be careful when we are referencing ~\cite[Lemma 3.2(4)]{ETS}, as this part of the lemma relied on the whole projection surface as opposed to just a small (local) portion of it. So now we will prove ~\cite[Lemma 3.2(4)]{ETS} for the two cases in a triangle-square pair .

Before we begin, there is something that must be noted. When these lemmas are applied in ~\cite{ETS}, they are to families of adjacent bigons. As such, the interior of these bigons will not have any vertices---if not, then any interior vertex must have a blue edge extending to either vertex of the bigon, breaking the fact that we have adjacent bigons. So the condition that the triangle-square pairs are disjoint from vertices holds for the ones we care about, allowing us to apply the following two lemmas in place of ~\cite[Lemma 3.2(4)]{ETS}:

\begin{lem}
Suppose that, in a triangle-square pair with interior disjoint from vertices, two blue edges, $x$ and $y$, meet such that the other endpoint of $x$ is the vertex of the triangle-square pair, and $y$ is in the same region as the crossing circle associated to the vertex. Then the crossing corresponding to this endpoint must be associated to the crossing circle.
\end{lem}

\begin{proof}
First, note that, by Lemma ~\ref{lem:brrect}, our triangle-square pair bounds a disk on $F$. Let $z$ be a blue edge adjacent to $x$ that also has an endpoint on the vertex, and $w$ the red edge from $x$ to $z$. Then, as $x$ and $z$ must be in distinct blue regions, and $x$ and $y$ must be in distinct blue region, we must have that $y$ and $z$ are in the same blue region. 

There is a simple closed curve connecting the crossing at the endpoint of $x$ and $y$ and the crossing circle. Let $\gamma$ be the curve that follows $x$ from the crossing to the crossing circle, goes through the crossing circle to $z$, and then jumps from $z$ to $y$ to make it back to the crossing. We claim that $\gamma$ bounds a disk. If $\gamma$ is contained entirely within the disk corresponding to the triangle-square pair, then we are done. Otherwise, at some point, we would have to leave this disk. By construction, the only time we could have done this is the jump from $z$ to $y$. That would mean that, although $z$ and $y$ are in the same blue region, we have to leave the triangle-square pair to realize this.
 
While this could be a problem, we are going to show that we can add to our triangle-square pair disk, and get a disk that includes all of $\gamma$. Let $\alpha$ be the simple closed curve that follows the jump from $z$ to $y$, then the red edge $w$ back to $z$, and then follows $z$ back to where it started.

\begin{figure}[!htb]
	\centering
	\def\svgwidth{.75\columnwidth}
	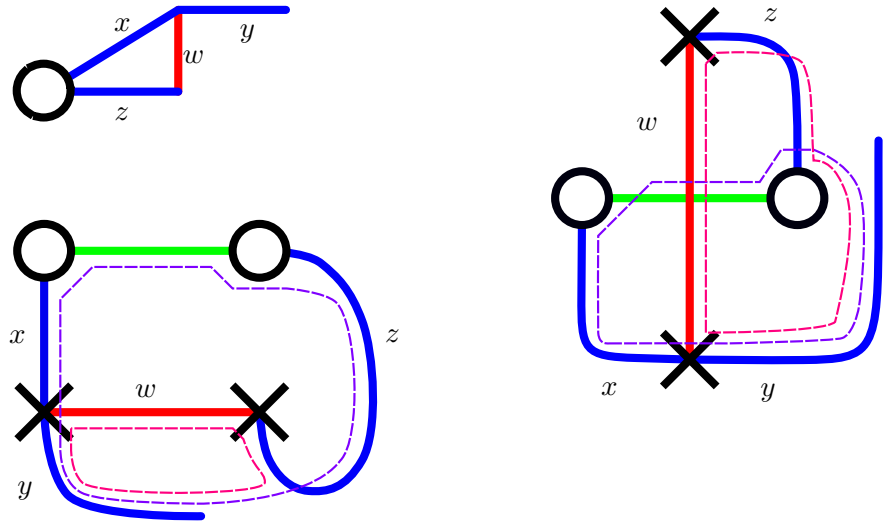
	\caption{\textit{(Upper Left)} The graph of the case we are considering. \textit{(Lower Left and Right)} The possible diagram configurations of this case. The purple curve following the crossing disk is $\gamma$, while the pink curve following $w$ is $\alpha$.}
\end{figure}

Then $\alpha$ must intersect our knot in exactly two places, and so bounds a disk. Further, this disk must have two of its edges entirely within our triangle-square pair disk, and the other edge exiting it. This will allow us to glue the $\alpha$ disk to the triangle-square pair disk, and still get a disk. This, then, will give us $\gamma$ living inside a disk, and so then must bound a disk, and we are done by definition.
\end{proof}

\begin{lem}
\label{lem:324subTriangles}
Suppose that, in a group of triangle-square pairs with interior disjoint from vertices, two blue edges, $x$ and $y$, both have an endpoint at the vertex, and in the image of the knot, have their other endpoint on opposite sides of the same crossing. Then that crossing is associated to the crossing circle of the vertex.
\end{lem}

\begin{proof}
As with above, we can use Lemma ~\ref{lem:brrect} to get a disk that the triangle-square pair lives on. In particular, both $x$ and $y$ live on the disk, as does the green edge associated with the crossing circle. Also, as $x$ and $y$ are on opposite sides of a crossing, they must also be on opposite sides of the crossing circle associated to the vertex. So we will create a simple closed curve that intersects our knot twice at the crossing circle and twice at the crossing of $x$ and $y$ by taking $\gamma$ to be the curve that follows $x$ from the crossing to the crossing circle, then the green edge across, then $y$ back down to the crossing. As $x$ and $y$ live in distinct blue regions, they cannot intersect except at the crossing, so $\gamma$ is simple. And, as $\gamma$ lives in a disk, it must also then bound a disk, and we are done.
\end{proof}

Now, the proofs of Lemmas 5.1 through 5.5 of ~\cite{ETS} will all work as intended, where we use the above two lemmas in place of ~\cite[Lemma 3.2 (4)]{ETS}. To summarize, these lemmas tell us how three adjacent triangle-square pairs must be placed in our diagram. Of particular note, the last three of these, Lemmas 5.3 through 5.5, state that, of the four red-blue crossings of the triangles, three of them must correspond to the vertex of the crossing circle, and they cannot all be adjacent to each other. That is, either the two top crossings and the bottom crossing are associated to the vertex, or the bottom two and the top crossing are. 

Now, by starting with three adjacent triangle-square pairs, we can start adding additional triangle-square pairs, and keep track of if the crossings are associated to the vertex of the crossing circle or not. As it turns out, as illustrated in ~\cite[Lemma 5.7]{ETS}, we get a contradiction when we try to add a fifth triangle-square pair, thus proving Lemma ~\ref{lem:fivepairs}. Using this, we can then get:

\begin{prop}
\label{lem:NoFiveBigons}
The graph $\Gamma_B$ cannot contain five or more adjacent non-trivial bigons.
\end{prop}
\begin{proof}
Suppose $\Gamma_B$ contains more than five adjacent non-trivial bigons. Then, by Lemma ~\ref{lem:BBigIntRorG}, the bigons cannot be disjoint from green and red. If the bigons are disjoint red, then they must intersect a green edge. However, as we cannot have blue-green bigons or blue-blue-green triangles (~\cite[Lemmas 4.4, 4.16]{ETS}), so the green edge must go from one vertex to another. If the green edge meets only one vertex, we have a green monogon, which cannot happen (~\cite[Lemma 4.16]{ETS} says we can remove such monogons), while if it meets two distinct vertices of our bigon, then the blue edges must be trivial by definition, a contradiction. So the blue bigons must meet the red surface.

As there are no blue-red bigons (Lemma ~\ref{lem:BlueRedBigon}), any red edge must run straight through all five bigons, giving us a collection of blue-blue-red triangles. By Lemma ~\ref{lem:adjtri}, we can't have three or more adjacent such triangles, so another red edge must run parallel, giving us adjacent triangles and squares instead. However, by Lemma ~\ref{lem:fivepairs}, we cannot have five adjacent triangle-square pairs, so we get a contradiction, and thus $\Gamma_B$ cannot have five adjacent non-trivial bigons.
\end{proof}

\section{Completing the Proof of Theorem 2.3}
Now we can get that the map $f:S_{B,2}\to Y\setminus K$ is $\pi_1$-injective:

\begin{thm}
\label{thm:pi_injective}
Let $f:S_{B,2}\to Y\setminus K$ be the immersion of $S_{B,2}$ into $Y\setminus K$. Then this immersion is $\pi_1$-injective, provided $N_{tw}\ge 121$.
\end{thm}

\begin{proof}
If not, then we can get a disk $\phi:D\to Y\setminus K$, such that $\phi|_{\bndry D} = f\circ \ell$, for some essential loop $\ell$, and graph $\Gamma_B=\phi^{-1}(f(S_{B,2})$, with vertices corresponding to crossing circles. These vertices have valence $2n_j$ if in the interior, or $n_j+1$ in the boundary, where $2n_j$ is the number of crossings removed from the twist region when constructing $S_{B,2}$. In particular, $2n_j \ge 120$. By Lemma ~\ref{lem:minbig}, as the minimum number of crossings removed is $R_{tw} \ge 2\lceil N_{tw}/2\rceil -2 \ge 120$, we must then have more than five adjacent non-trivial bigons. This, however, contradicts Lemma ~\ref{lem:NoFiveBigons} above, and so $f$ must be $\pi_1$-injective.
\end{proof}

Likewise, the proof of boundary-$\pi_1$-injectivity can be proved as in ~\cite{ETS}. As before, instead of reproving every lemma we need, we will instead prove ~\cite[Lemma 3.2(4)]{ETS} in the specific case we need, and replace the uses of Lemma 3.2(4) with this new one:

\begin{lem}
Suppose a triangle with two blue edges and one edge on $\bndry N(K)$ does not meet the red edge. Then the two blue edges must be on opposite sides of a crossing, and that crossing must be associated to the crossing circle of the vertex.
\end{lem}

\begin{proof}
As the two blue edges of the triangle meet at the same vertex, they must map opposite sides of the crossing circle. However, as the third side is the knot strand, and the triangle does not meet red, the strand must run through a single crossing, proving the first part. To see that the crossing is associated to the crossing circle, we can construct a closed curve intersecting the knot exactly four times: Take one of the blue edges from the crossing circle down to the crossing, then follow the other blue edge up to the crossing circle, and pass over to the start through the green edge. We know immediately that this curve cannot be a compressing disk---$e(\pi(K_{B,2}),F)>4$. So the other option is that it is essential. However, if this curve is essential and not a compression disk, it must pass through the projection surface. It can only do this where red and blue meet. As the triangle is disjoint from red, this cannot happen, so the curve must bound a disk on the projection surface. But then we have a disk whose boundary passes only through the crossing in question and the crossing circle, so the crossing is related to the crossing circle, and we are done.
\end{proof}

From here, the remaining lemmas follow nicely. To summarize, first, by looking at how blue-blue-$\bndry N(K)$ triangles can be placed in the diagram, we get that, if we have three adjacent triangles in $\Gamma_B$, they must meet the red surface. This is the lemma that references ~\cite[Lemma 3.2(4)]{ETS}, and so we can use the above lemma in place of it. Further, by careful examination, ~\cite[Lemma 6.2]{ETS} tells us that three adjacent such triangles must meet the red surface at least twice. This, then, gives us blue-red triangle square pairs, and we can use Lemma ~\ref{lem:fivepairs} to say that we can't have five adjacent triangle-square pairs, and thus, can't have five adjacent blue-blue-$\bndry N(K)$ triangles.

Now, we can prove the $\bndry-\pi_1$-injective:

\begin{thm}
\label{thm:bndry_pi_injective}
Let $f:S_{B,2}\to Y\setminus \mathrm{int}(N(K))$ be the immersion of $S_{B,2}$ into $Y\setminus \mathrm{int}(N(K))$. Then this immersion is $\bndry-\pi_1$-injective, provided $N_{tw}\ge 121$.
\end{thm}

\begin{proof}
If not, then we can get a disk $\phi:D\to Y\setminus \mathrm{int}(N(K))$, such that $\bndry D$ is a concatenation of two arcs, one mapped into $\bndry N(K)$, and the other an essential arc in $S_{B,2}$ ~\cite[Lemma 2.2]{ETS}. By Lemma ~\ref{lem:mintri}, as the minimum number of crossings removed is $R_{tw} \ge 2 \lceil N_{tw}/2 \rceil -2 \ge 120$, we must have more than five adjacent non-trivial bigons or five adjacent triangles with an edge on $\phi^{-1}(\bndry N(K))$. However, as mentioned above, by ~\cite[Lemma 6.3]{ETS}, we cannot have five adjacent such triangles. Also, by Lemma ~\ref{lem:NoFiveBigons}, we cannot have five adjacent non-trivial bigons. This gives us a contradiction, so $f$ must be $\bndry-\pi_1$-injective.
\end{proof}

Putting Theorems ~\ref{thm:pi_injective} and ~\ref{thm:bndry_pi_injective} together, we get:

\essentialthm*

\begin{proof}
We first prove that $f:S_{B,2}\to Y\setminus K$ is $\pi_1$-injective. If not, then, by ~\cite[Lemma 2.1]{ETS} (and the remarks at the beginning of section 5 about why this carries over to our general case), we get a map of a disk $\phi:D\to Y\setminus K$ with $\phi|_{\bndry D} = f\circ \ell$ for some essential loop $\ell$ in $S_{B,2}$, with $\Gamma_B = \phi^{-1}(f(S_{B,2}))$ a collection of embedded closed curves and an embedded graph in $D$. Each vertex in the interior of $D$ has valence a non-zero multiple of $2n_j > 121$, where $2n_j$ is the number of crossings removed, and each vertex in the exterior has valence $n_j+1 > 60$. Then, as we are working with a graph with exactly the same properties as the graph in Section 2 of ~\cite{ETS}, we can get that $\Gamma_B$ has more than $(N_{tw}/18)-1>5$ adjacent non-trivial bigons. 
\end{proof}

\subsection{Properties of Twisted Surfaces}

Now that we have proven that our twisted surfaces are essential surfaces in $Y\setminus K$, we can get a bit more out of them as well. Ultimately, we want to show the following:

\homotopicarcs*

This theorem is proven in a series of lemmas, all but one of which require little to no alteration in statement and proof. We will go through each of the lemmas, and give a sketch of why they still hold for our case. Recall that \textit{the subsurface associated to a twist region of $K_2$} is the intersection of a surface with a regular neighborhood of the twist region.

First in our proof of Theorem ~\ref{thm:homarcs} is to get a disk to work on. 

\begin{lem}
\label{lem:ETS7.2}
Suppose homotopically distinct essential arcs $a_1$ and $a_2$ in $S_{B,2}$ map by $f:S_{B,2}\to Y\setminus K$ to homotopic arcs $e_1$ and $e_2$ in $Y\setminus K$. Then there is a map of a disk $\phi:D\to Y\setminus \mathrm{int}(N(K))$ with $\bndry D$ expressed as four arcs, with opposite arcs mapping by $\phi$ to $e_1$ and $e_2$, and the other two arcs mapping to $\bndry N(K)$.
Moreover, $\Gamma_B=\phi^{-1}(f(S_{B,2}))$ forms a graph on $D$ whose edges have endpoints either at vertices where $\phi(D)$ meets a crossing circle, or on $\phi^{-1}(\bndry N(K))$ on $\bndry D$. Each vertex has valence either a multiple of $2n$ (if in the interior of $D$), $n+1$ (if in the interior of an arc on $\bndry D$ that maps to $S_{B,2}$), or 1 (if on an arc that maps to $\bndry N(K)$), where $2n$ is the number of crossings removed from the twist region.
\end{lem}

\begin{proof}
As $e_1$ and $e_2$ are homotopic, we immediately get a disk $\phi:D\to Y\setminus \mathrm{int}(N(K))$ with the correct edges. Wee then modify this disk to get the remaining properties.

First, we lift $e_1$ and $e_2$ to the orientable double cover of $S_{B,2}$ By pushing $e_1$ and $e_2$ in the transverse direction, and using the fact that $\bndry N(K)$ is transversely orientable, we can get the map of $\phi$ on a neighborhood of $\bndry D$. Then we can extend $\phi$ over the rest of $D$, making it transverse to all crossing circles and to $f(S_{B,2}$.

Now look at $\Gamma_B = \phi^{-1}(f(S_{B,2}))$. As $S_{B,2}$ is embedded in $Y\setminus K$ except as crossing circles, $\Gamma_B$ must consist of embedded closed curves (where $S_{B,2}$ intersects $\phi(D)$ away from crossing circles or $\bndry N(K)$), arcs with endpoints on vertices (where $S_{B,2}$ intersects a crossing circle), and arcs with endpoints on $\bndry N(K)$.

As a vertex of $\Gamma_B$ in the interior of $D$ comes from where $\phi(D)$ intersects a crossing circle, the valence is equal to the number of crossings removed, by construction of the twisted surface $S_{B,2}$. If we have a vertex on an arc in $\bndry D$, it maps either to $S_{B,2}$ or $\bndry N(K)$. If it maps to $S_{B,2}$, then a neighborhood of the arc must map to half a meridian disk of a crossing circle, and so has valence $n+1$, where we removed $2n$ crossings from the corresponding twist region in our construction of $S_{B,2}$. If it maps to $\bndry N(K)$, however, because the arc must be transverse to $S_{B,2}$, the vertex must have valence 1. 
\end{proof}

We can also take our disk to be minimal, in the sense that it is a disk that gives us the fewest number of vertices of $\Gamma_B$ (the graph coming from how the disk intersects our blue surface) and $\Gamma_{BR}$ (the graph coming from how the disk intersects both our blue and red surfaces). Next, by examining graphs in a disk that satisfy the properties above, we get an analogue to Lemma ~\ref{lem:mintri}:

\begin{lem}
\label{lem:ETS7.5}
Suppose $\Gamma_B$ contains at least one blue vertex. Then either $\Gamma_B$ has more than $(R_{tw}/24)-1$ adjacent non-trivial bigons, or more than $(R_{tw}/24)-1$ adjacent triangles, each with one arc on $\phi^{-1}(\bndry N(K))$.
\end{lem}

The proof of this lemma is a graph theoretic proof, and so does not rely on the ambient space $Y$ or the twisted surfaces $S_{B,2}$ or $S_{R,2}$. 

\begin{sproof}
First, suppose we have a graph $\Gamma$ on $I\times I$ with no monogons and only valence one vertices on $I\times \bndry I$, with interior vertices with valence at least $R$ and vertices on $\bndry I\times (I-\bndry I)$ have valence at least $R/2+1$. Then $\Gamma$ must have more than $R/8 - 1$ adjacent bigons or more than $R/8-1$ adjacent triangles with an edge on $I\times \bndry I$. 

Next, we can modify our graph $\Gamma_B$ coming from Lemma ~\ref{lem:ETS7.2} by looking at subdisks, doubling along edges, and collapsing certain bigons and triangles so that we satisfy the above statement, with $R=R_{tw}$ is the minimal number of crossings removed from a twist region. When we collapsed bigons or triangles, we made it so at most three bigons or triangles in $\Gamma_B$ are associated to a bigon or triangle in this new graph. As such, we get $(R_{tw}/24)-1$ adjacent non-trivial bigons in $\Gamma_B$, or more than $(R_{tw}/24)-1$ adjacent triangles, and are done.
\end{sproof}

Using this, we can prove:

\begin{lem}[Lemma 7.6 ~\cite{ETS}]
\label{lem:ETS7.6}
If $N_{tw}\ge 121$, then the graph $\Gamma_B$ contains no blue vertices. That is, $\phi(D)$ meets no crossing circles.
\end{lem}

\begin{proof}
As shown in Lemma ~\ref{lem:nomono}, $\Gamma_B$ has no monogons. If $\Gamma_B$ has any trivial arcs in $\bndry D$, then that arc must bound a disk $E$ that is a subset of $S_{B,2}$. By sliding $\bndry D$ along $E$, we can remove such an arc, so we may assume that $\Gamma_B$ has no trivial arcs in $\bndry D$. So any blue vertex of $\Gamma_B$, so we can use Lemma ~\ref{lem:ETS7.5}, and get more than $((N_{tw}-1)/24)-1$ adjacent non-trivial bigons or adjacent triangles with an arc on $\phi^{-1}(\bndry N(K))$. When $N_{tw}\ge 121$, we get at least 5 such bigons or triangles. By Lemma ~\ref{lem:NoFiveBigons}, we cannot have five such bigons, and ~\cite[Lemma 6.3]{ETS}, we cannot have five adjacent triangles, and so $\Gamma_B$ cannot have any blue vertices. 
\end{proof}

So now we can start looking at how exactly the surfaces $S_{B,2}$ and $R_2$ (the red surface arising from $K_2$) intersects our disk. First, with no blue vertices, $\Gamma_B$ consists only of blue arcs with end points on the north and south edges of the disk, corresponding to $\bndry N(K)$. If both of the edges are on the same edge, by Theorem ~\ref{thm:bndry_pi_injective}, because the disk is $\bndry \pi_1$-injective, the blue edge is trivial and we can find a disk that doesn't meet that edge. So blue edges must run from north to south. Likewise, the red edges of $\Gamma_{BR}$ are only arcs that go from one distinct edge to another, from Lemma ~\ref{lem:BlueRedBigon} (no blue red bigons), and the fact that the red checkerboard surface $R_2$ is essential.

Now we need to show that the red arcs run along opposite edges, either east to west ($e_1$ to $e_2$), or north to south ($\bndry N(K)$ to $\bndry N(K)$):

\begin{lem}
\label{lem:ETS7.8}
The graph $\Gamma_{BR}$ consists of red and blue arcs with endpoints on opposite sides of $D$ (north-south or east-west). Further, blue arcs run north to south.
\end{lem}

\begin{proof}
In the case of blue edges, we are done, as we know the arcs run from north to south as discussed above. For red arcs, if they don't have endpoints on opposite sides, then they must form a triangle with one endpoint on either the north or south side, and the other on the east or west side. This gives us a blue-red-$\bndry N(K)$ triangle. The blue and red arcs must meet at a crossing of our arc, so, assuming the triangle region maps above the projection plane, there are two ways this could happen. Either the red edge runs from one crossing (where it meets the blue arc) to near another crossing, or runs from one crossing to $\bndry N(K)$, with the blue arc on the opposite side of this meeting.

\begin{figure}[!htb]
	\centering
	\def\svgwidth{.75\columnwidth}
	
	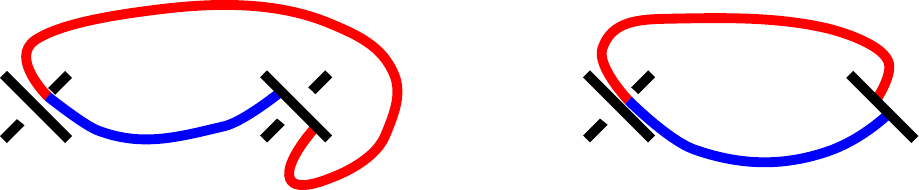
	\caption{Two ways a triangle with one red on red, one edge on blue, and one edge on the knot can sit in the diagram. \label{fig:rbk_tri}}
\end{figure}

Before we can use primality arguments to show neither of the cases can happen, we need to show that we can get a disk from this triangle. By joining the blue and red arcs together at the endpoints that don't meet, either by sliding along $K$ or connecting across a crossing, we get a closed curve that intersects our knot diagram twice. As $e(\pi(K),F)\ge 4$, this curve cannot be essential, and so must bound a disk on $F$.

From here, the remainder of the proof follows as in ~\cite{ETS}. If we are in the left case of Figure ~\ref{fig:rbk_tri}, connecting the endpoints will give us a blue-red bigon. But then, as in the proof of Lemma ~\ref{lem:BlueRedBigon}, we can get a contradiction to primality. On the other hand, if we are in the other case, because $K_{B,2}$ is weakly prime, there must be no crossings in the interior of the closed curve we constructed. As such, both the red and blue arcs must be homotopic to a portion of $\bndry N(K)$. We can then use a homotopy to slide our disk $D$ away from where the blue and red arc intersect. This will give us one less vertex on $D$, and so contradicts the minimality of $D$. In either situation, we get a contradiction, so the red arcs must run either north-south or east-west.
\end{proof}

As red arcs cannot intersect red arcs, we must also have that all red arcs run in the same direction---either all north-south or all east west. We consider the first case now:

\begin{lem}
\label{lem:ETS7.9}
If the graph $\Gamma_{BR}$ cuts $D$ into a subrectangle with two opposite sides mapped to $\bndry N(K)$, one side on blue, one side on red, and the interior disjoint from blue and red, then the blue and red sides of that rectangle are homotopic to the same crossing arc, and the homotopies can be taken to lie entirely in $S_{B,2}$ and $R_2$.
\end{lem}

\begin{proof}
First, note that such a rectangle, with interior disjoint from both blue and red, must be mapped either completely above or below the projection surface. Suppose it is above. Then there are three cases to consider. First, the blue and red arcs could straddle over crossings at both ends. Second, one set of endpoints could straddle an over crossing, while the other set lies on opposite sides of a strand. Third, both sets of endpoints lie on opposite ends of strands.

\begin{figure}[!htb]
	\centering
	\def\svgwidth{.75\columnwidth}
	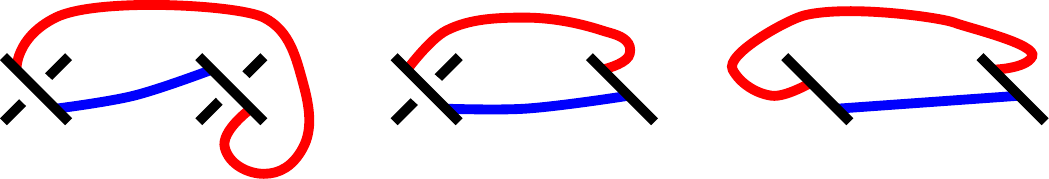
	\caption{The three ways a red-blue rectangle can sit in our diagram}
\end{figure}

Note that in all three of these cases, by connecting the red and blue arcs, we get a curve that intersects our knot exactly twice and so, as before, must bound a disk on $F$. The first case, with two over crossings, as in Lemma ~\ref{lem:BlueRedBigon} (there are no blue-red bigons), violates the fact that our diagram is Weakly Prime. The third case, with two strands, can be homotoped away from as in Lemma ~\ref{lem:ETS7.8}. That leaves us with just our second case. In this case, both the blue and red arc are homotopic to the same crossing arc (the over crossing they straddle), and the homotopies lie entirely in their corresponding surfaces, either $S_{B,2}$ or $R_2$, so we are done.
\end{proof}

From here, with no change to proof except to use $e(\pi(K),F)\ge 4$ to show the curves in question bound a disk, by carefully examining how rectangles with two colored sides (either both blue, both red, or blue and red) and two $\bndry N(K)$ sides can give us homotopies, we get the following result:

\begin{lem}
\label{lem:ETS7.10}
If the graph $\Gamma_{BR}$ consists of disjoint red and blue arcs on $D$, all running north to south, then $e_1$ and $e_2$ are each homotopic in the blue surface to arcs in the same subsurface associated with a twist region of $K_2$.
\end{lem}

\begin{sproof}
Our arcs split our disk into several sub-rectangles, with east and west sides colored either red or blue, and the north and south side on $\bndry N(K)$. If both colored sides are blue, the subrectangle must be mapped to a single side of the projection surface, and the north and south sides must run over crossings (or else we could homotope $D$ away to remove intersections). This means our subrectangle gives us a simple closed curve that meets our diagram at two crossings with the subrectangle on only one side of the projection surface, and thus gives us a disk. Because $K_{B,2}$ is weakly blue-twist reduced, this disk bounds a collection of red bigons, and so we can isotope the blue arcs to lie in the same subsurface. In a similar fashion, having to be a bit more careful as we don't know $K_{B,2}$ is weakly red-twist reduced, we get the same result for when both sides are red.

If the two sides are colored differently (one red, one blue), then Lemma ~\ref{lem:ETS7.9} allows us to homotope the two sides into the same subsurface. Putting it all together, two successive arcs are homotopic into the same subsurface, in either $S_{B,2}$ or $R_2$, as appropriate. Expanding this out, we get that $e_1$ and $e_2$ must be homotopic in the blue surface into the same subsurface.
\end{sproof}

For the last lemma, we need an analogue to the polyhedron decomposition of alternating knots. We will use the chunk decomposition of Howie and Purcell ~\cite{WGA}:

\begin{defn}
A \textit{chunk} $C$ is a compact, oriented, irreducible 3-manifold with boundary $\bndry C$ containing an embedded non-empty graph with all vertices having valence at least 3. The graph separates $\bndry C$ into regions called \textit{faces}. If the region is disjoint from the graph, it is called an \textit{exterior face}, otherwise, it is an \textit{interior face}.

A \textit{truncated chunk} is a chunk where a regular neighborhood of each vertex of the edge graph has been removed. This leaves a \textit{boundary face} surrounded by \textit{boundary edges}.

A \textit{chunk decomposition} for a 3-manifold $M$ is a decomposition of $M$ into chunks, such that $M$ is obtained by gluing chunks by homeomorphisms of non-exterior, non-boundary faces with edges mapping to edges homomorphically.
\end{defn}

Howie and Purcell showed in Proposition 3.1 in the same paper that weakly generalized alternating knots admit a chunk decomposition:

\begin{thm}
\label{thm:WGA_ChunkDecomp}
Let $Y$ be a compact, orientable, irreducible 3-manifold with no boundary components, containing a generalized projection surface $F$. Let $K$ be a weakly generalized alternating knot with a diagram $\pi(K)$ on $F$. Then $Y\setminus K$ can be decomposed into pieces such that:
\begin{enumerate}
	\item Pieces are homeomorphic to components of $Y\setminus N(F)$, except each piece has a finite set of points removed from $\bndry(Y\setminus N(F))$, namely the ideal vertices below.
	\item On each copy of $F$, there is an embedded graph with vertices, edges, and regions identified with the diagram graph $\pi(K)$. All vertices are ideal and 4-valent.
	\item To obtain $Y\setminus K$, glue pieces as follows. Each region of $F\setminus \pi(L)$ is glued to the corresponding region on the opposite copy of $F$ by a homeomorphism that is the identity composed with a rotation along the boundary. The rotation takes an edge of the boundary to the nearest edge in the direction of that boundary component's orientation.
	\item Edges correspond to crossing arcs, and are glued in fours. At each ideal vertex, two opposite edges are glued together.
\end{enumerate}
\end{thm}

With this in mind, we can generalize the following lemma, initially proven by Lackenby ~\cite{CVAK}:

\begin{lem}
\label{lem:EsntlSqrIntersect}
Let $C$ be a chunk in a chunk decomposition for $Y\setminus K$, where $K$ is a weakly generalized alternating knot on $F$ with $e(\pi(K),F)\ge 4$. Then, if $S$ and $T$ are essential squares, isotoped to minimize $|S\cap T|$, they intersect either zero or two times. If they intersect two times, then the points of intersection lie in distinct regions of the diagram with the same color.  
\end{lem}

\begin{proof}
We can color the faces of our chunk blue or red, depending on how it meets the checkerboard coloring of $F\setminus K$. An essential square intersects the knot projection at four distinct points, and so the square must have two blue edges and two red edges, each opposite each other. As $S$ and $T$ have been isotoped to have minimal intersection, there are at most four points of intersection, one at each edge. If there are four points of intersection, then $S$ and $T$ are isotopic to each other, and can be pushed off to have no intersection. In addition, they cannot intersect an odd number of times, as both bound disks on $F$. If they do, look at how $T$ as a curve intersects $S$. It must cross the boundary of $S$ exactly three times. As $T$ bounds a disk on $F$, however, this is impossible. So $S$ and $T$ intersect either zero or two times.

Suppose they intersect exactly twice, and one of those intersection points happens on red, and the other on blue. Because $S$ and $T$ must intersect the knot where blue meets red, if we look at the arcs of $S\setminus T$ and $T\setminus S$, they must intersect our knot either once or three times. Look at the two arcs that intersect the knot once, each, and glue them together to get a curve intersecting our knot exactly twice. As $e(\pi(K),F)\ge 4$, this curve must bound a disk on $F$. Then, as $K$ is weakly prime, because the curve intersects the knot exactly twice and bounds a disk, the disk must intersect our knot in a single arc with no crossings. We can then use this arc to homotope $S$ and $T$ away from each other, contradicting the fact that our intersection was minimal. So, if $S$ and $T$ intersect, they must intersect in distinct regions of the same color.
\end{proof}

As a consequence to this, we get the following result:

\begin{lem}
\label{lem:EsntlSqrVertex}
Let $C$ be a chunk in a chunk decomposition for $Y\setminus K$, where $K$ is a weakly generalized alternating knot on $F$ with $e(\pi(K),F)\ge 4$. Let $S$ and $T$ be essential squares in $P$, moved by normal isotopy to minimize $|S\cap T|$. Suppose $S$ and $T$ pass through the same red face $W$, and that edges $S\cap W$ and $T\cap W$ differ by a single rotation of $W$. Then exactly one of the following two conclusions holds:
\begin{enumerate}
	\item Each of $S$ and $T$ cuts off a single ideal vertex in $W$, and $S$ and $T$ are disjoint.
	\item Neither $S$ nor $T$ cuts off a single vertex in $W$. The two essential squares intersect in $W$ and in another face $W'$.
\end{enumerate}
\end{lem}

\begin{proof}
First, if $S$ cuts off a single ideal vertex, then so must $T$, and so do not intersect in $W$. By Lemma ~\ref{lem:EsntlSqrIntersect}, then, if $S$ and $T$ intersect, it must be in a blue face. As $e(\pi(K),F)\ge 4$, $W$ can only meet a given blue face once (otherwise we could find an essential curve on $F$ intersecting our knot exactly twice). Then, if $S$ and $T$ do intersect in a blue face (and thus in two blue faces), then $S\cap W$ and $T\cap W$ must be parallel to each other. We could then isotope $S$ and $T$ away from each other in the blue faces so that $S$ and $T$ are disjoint.

If $S$ does not cut off a single vertex, then neither does $T$. As $S\cap W$ and $T\cap W$ differ by a single rotation, they must intersect in $W$. By Lemma ~\ref{lem:EsntlSqrIntersect}, then, they must intersect in two red faces, $W$ and $W'$, and so we are done.
\end{proof}

We can use this to get our final needed lemma:

\begin{lem}
\label{lem:ETS7.12}
Let $e_1$ and $e_2$ denote the boundary arcs on the blue surface in an essential product disk for the blue checkerboard surface of $K_{B,2}$. Then there is a subsurface associated with a twist region of the diagram of $K_2$, and arcs $a_1$ and $a_2$ in that subsurface, such that $e_1$ is homotopic in $S_{B,2}$ to $a_1$, and $e_2$ is homotopic in $S_{B,2}$ to $a_2$.
\end{lem}

\begin{proof}
Let $E$ be the essential product disk. If it is disjoint from the red surface, then the blue edges of $E$ are homotopic to arcs in a subsurface associated to one twist region of $K_{B,2}$. If $E$ does meet the red surface, then, by Lemma ~\ref{lem:ETS7.8}, either the red edges all run north-south ($\bndry N(K)$ to $\bndry N(K)$) or east-west (from blue edge to blue edge). In the former case, Lemma ~\ref{lem:ETS7.9} gives us our result. So we may assume we are in the latter case, and red edges run from blue edge to blue edge, cutting the rectangle into blue-red-blue-red sub-rectangles.

\begin{figure}[!htb]
	\centering
	\def\svgwidth{.8\columnwidth}
	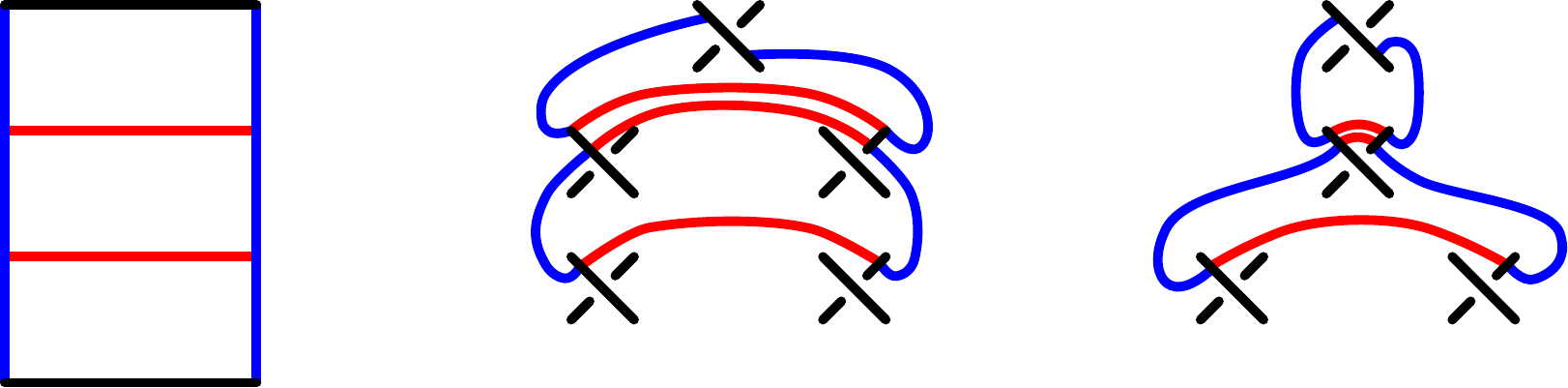
	\caption{\textit{(Left)} An essential product disk with blue sides $e_1$ and $e_2$ intersected horizontally by red edges. \textit{(Right)} Two sub-rectangles whose shared red edge meets two separate crossings. \textit{(Center)} Two sub-rectangles whose share red edge meets a single crossing. We want to prove that this is what happens for all adjacent sub-rectangles.}
\end{figure}

By Theorem ~\ref{thm:WGA_ChunkDecomp}, we get a chunk decomposition for our knot, with pieces homeomorphic to components of $Y\setminus N(F)$, except for ideal vertices corresponding to crossings. Each face of the chunk can be colored either red or blue, depending on what color it meets on $F$. As $F$ is a connected surface, $Y\setminus N(F)$ has either one or two components. Each of our sub-rectangles that make up $E$ must be contained entirely in a single chunk. Further, if two rectangles glue to each other along a red edge, the rectangles are either not in the same chunk, if we have two chunks, or meet two separate red faces that glue to each other, if there is only one chunk. We proceed based on whether we have only one chunk or two.

If we only have a single chunk, then look at the first two rectangles in $E$, $E_1$ the north most one that meets $\bndry N(K)$, and $E_2$, the one below that which glues to the red edge, $r$ of $E_1$. Looking at the side of $E_1$ that meets $\bndry N(K)$, it must run through an ideal vertex of the chunk decomposition, and so we can push it off into the red face so that it cuts off a single vertex. Consider the other side of $E_1$ now, which glues by a clockwise or counterclockwise turn to $E_2$. While the red side of $E_1$ and the red side of $E_2$ are not on the same face, we can create a new rectangle $\bar{E_2}$ by copying $r$ in $E_2$ into the same face of $r$ in $E_1$, after rotating by the corresponding turn, and then copying the rest of $E_2$ over. By Lemma ~\ref{lem:EsntlSqrIntersect}, if the two copies of $r$ (for $E_1$ and $\bar{E_2}$) intersect, then $E_1$ and $\bar{E_2}$ must also intersect in the other red face. However, as the other edge of $E_1$ cuts of an ideal vertex, we can homotope it so it does not intersect the other edge of $\bar{E_2}$. This means, then, that $E_1$ and $\bar{E_2}$ cannot intersect in a red face. By Lemma ~\ref{lem:EsntlSqrVertex}, then, both $r$ in $E_1$ and in $\bar{E_2}$ must cut off a single vertex, and so $r$ in $E_2$ must also cut off a single vertex. Continuing by induction, each rectangle making up $E$ must have sides in the red faces cutting off a single vertex.

Now suppose we have two chunks. Once again, we start by looking at the first two rectangles of $E$, $E_1$ and $E_2$, which glue by a single turn along a red edge $r$. Like before, we can homotope the north end of $E_1$ so that it cuts off a single vertex. As $E_1$ and $E_2$ are in separate chunks, we can impose $E_2$ into the same chunk as $E_1$ by copying $r$ in $E_2$ into the chunk containing $E_1$ under a clockwise turn, and continuing for the rest of $E_2$. We call this new rectangle $\bar{E_2}$. With the same argument as in the single chunk case, we get that $r$ of both $E_1$ and $E_2$ must cut off a single vertex, and, by induction, each red edge of $E$ must cut off a single ideal vertex.

In either case, each rectangle of $E$ maps to a region of the diagram meeting $\pi(K_{B,2})$ exactly four times, adjacent to two crossings (one for each red edge). As $\pi(K_{B,2})$ is weakly blue twist reduced by Lemma ~\ref{lem:wBluTwistReduced}, this region must bound a string of red bigons, and so the boundaries all lie in the same twist region. By applying this to all the rectangles in $E$, we get that the two blue edges must lie in a neighborhood of the same twist region, and so are done.
\end{proof}

This is done by looking at how the red-blue-red-blue rectangles lie in the polyhedral decomposition induced by the checkerboard surfaces. By using a chunk decomposition instead, we can get the same result in our general case. In particular, this tells us that, we can find arcs $a_1$ and $a_2$ in the same subsurface of $K_2$ such that $e_1$ and $e_2$ are homotopic to $a_1$ and $a_2$, respectively, in $S_{B,2}$. Putting this together with the above lemma, we get our main result for the section.

\begin{proof}[Proof of Theorem ~\ref{thm:homarcs}]
Call our two arcs $e_1$ and $e_2$. Then, by Lemma ~\ref{lem:ETS7.2}, we get a disk $\phi:D\to Y\setminus K$ and a graph $\Gamma_B$. Then, by Lemma ~\ref{lem:ETS7.6}, there are no blue vertices in our graph. Next, Lemma ~\ref{lem:ETS7.8} tells us that blue edges must run north to south, while red edges either run all north to south or all east to west. If they all run north to south, then Lemma ~\ref{lem:ETS7.10} tells us $e_1$ and $e_2$ are homotopic into the same subsurface. If the red edges all run east to west, then Lemma ~\ref{lem:ETS7.12} gives us the same result, and we are done.
\end{proof}

\bibliography{WGA_final}{}
\bibliographystyle{amsplain}
\end{document}